\documentclass[12pt]{amsart}

\usepackage[margin=1in]{geometry}

\usepackage{amsmath}
\usepackage{comment}
\usepackage{listings}
\usepackage{graphicx}
\usepackage{bbm}
\usepackage[export]{adjustbox}
\usepackage{xy}
\usepackage{tikz}
\usetikzlibrary{shapes}
\usetikzlibrary{backgrounds}
\usetikzlibrary{trees}
\usetikzlibrary{intersections}
\usetikzlibrary{knots}
\usepackage{tikz-qtree}
\usepackage{relsize}

\newcommand{\Rep}{\text{Rep}}

\newcommand{\xdiag[1]}{\begin{tikzpicture}[baseline = -.05cm]
\draw[fill = black!40!, draw = none] (-.35,-.5) rectangle (.35,.5);
\draw (-.35,-.5) -- (-.35,.5);
\draw (.35,-.5) -- (.35,.5);
\draw[fill = white] (-.5,-.3) rectangle (.5,.3);
\node at (0,0) {\large{$#1$}};
\node at (-.7,0) {$a$};
\node at (.65,0) {$b$};
\end{tikzpicture}}

\newcommand{\xdiagscaled[1]}{\begin{tikzpicture}[baseline = -.1cm, scale = .6]
\draw[fill = black!40!, draw = none] (-.35,-.5) rectangle (.35,.5);
\draw (-.35,-.5) -- (-.35,.5);
\draw (.35,-.5) -- (.35,.5);
\draw[fill = white] (-.5,-.3) rectangle (.5,.3);
\node at (0,0) {\tiny{{$#1$}}};
\end{tikzpicture}}

\newcommand{\skeincircle}{\begin{tikzpicture}[baseline]
\draw circle (.5cm);
\end{tikzpicture}}

\usepackage{amssymb}
\usepackage{epstopdf}
\usepackage{amsthm}
\usepackage{enumitem}
\usepackage{multirow}
\newtheorem{theorem}{Theorem}[subsection]
\newtheorem{customthm}{Theorem}
\newtheorem{corollary}[theorem]{Corollary}
\newtheorem{lemma}[theorem]{Lemma}
\newtheorem{proposition}[theorem]{Proposition}

\theoremstyle{definition}
\newtheorem{definition}[theorem]{Definition}
\newtheorem{example}[theorem]{Example}

\theoremstyle{remark}
\newtheorem*{remark}{Remark}

\lstdefinestyle{mystyle}{
basicstyle=\times}
\title{Spin models for singly-generated Yang-Baxter planar algebras}
\author{Joshua R. Edge}
\address{Department of Mathematics and Computer Science, Denison University, 100 W. College St., Granville, OH 43023}
\email{edgej@denison.edu}
\date{July 29, 2020}

\begin{document}

\begin{abstract}
In this paper, we classify all spin models for singly-generated Yang-Baxter planar algebras in terms of certain highly symmetric graphs. Using Liu's classification of singly generated Yang-Baxter planar algebras, this classifies all spin models for the Jones polynomial, the Bisch-Jones planar algebras, and the Kauffman polynomial. This simplifies and clarifies Jaeger's classification of spin models for the Kauffman polynomial.  In particular, we completely avoid the opaque and computationally-taxing concept of ``formal self-duality"; we also explain the numerous exceptional cases in Jaeger's classification by demonstrating a new, discrete two-parameter family of spin models for the Bisch-Jones planar algebras.
\end{abstract}

\maketitle

\section{Introduction}

In \cite{Jon85}, we see the first instance of the Jones polynomial, the first knot polynomial to be discovered since the Alexander polynomial in 1923. The Kauffman polynomial was first defined a few years later in \cite{Kau90}. At around the same time, Kauffman also noticed a curious connection between the Jones polynomial and spin models, which are physical models used to explain magnetism. In \cite{Jae95}, Jaeger noted a similar relationship between the Kauffman polynomial and spin models. 

Given a tangle, the regions of it can be shaded in a checkerboard fashion. Up to rotation, a crossing can only be shaded in two ways:
\begin{center}
\includegraphics[valign = c]{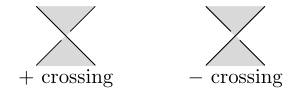}
\end{center}
The original definition of a spin model for a knot polynomial involved choosing a finite set $S$ and symmetric maps $w_{\pm}:S\times S \to \mathbb{C}.$  If we label the two unshaded regions of the $+$ crossing with elements of $S,$ we think of $w_+$ as assigning weights to each diagram:
\begin{center}
\includegraphics[valign = c]{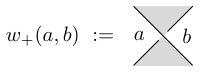}
\end{center}
$w_-$ gives weight for the $-$ crossing. We often think of these maps as $|S|\times|S|$ matrices, $W_\pm$. Together the triple $M = (S,w_+, w_-)$ give a state sum model, which allows us to assign numbers to knots or links in the following manner. Given a knot or link, $K$, we assign atoms of $S$ to the unshaded regions of $K$, which we call a state, $\sigma,$ of $K$. $\sigma,$ then, assigns a weight to each crossing of $K,$ and we can obtain a number for $\sigma$, $w_\sigma,$ by taking the product of these weights. A state sum for $K$ is $\sum_{\sigma} w_\sigma$ for all possible states, $\sigma$ of $K.$
 
A spin model for a knot polynomial is a state sum model that is invariant under the relations of the knot polynomial. That is, if one knot or link can be obtained from another via the Reidermeister moves or some other relation of the knot polynomial, then the number assigned to them by the state sum model will be the same. 

In \cite{Kau87}, Kauffman showed that $W_+ = t_0Id + t_1 A_1,$ where $t_i \in\mathbb{C}$ and $A_1$ is the adjacency matrix of a complete graph. Similarly in \cite{Jae95}, Jaeger showed that $W_+ = t_0Id + t_1 A_1 + t_2 A_2,$ where $t_i \in\mathbb{C}$, $A_1$ is the adjacency matrix of some graph $\Gamma$ and $A_2$ is the adjacency matrix of $\Gamma^c,$ the graph complement of $\Gamma.$ He then gave necessary and sufficient conditions on $\Gamma$ for a specialization of the Kauffman polynomial to have a spin model, which can be stated as follows:
\begin{customthm}[\cite{Jae95}]
Let $\mathcal{V}$ be a specialization of the Kauffman polynomial. Then $\mathcal{V}$ admits a spin model if and only if $A_1$ is the adjacency matrix of some graph $\Gamma$ of $n$ vertices with the following properties:
\begin{enumerate}[label = \roman{enumi}.]
    \item $\Gamma$ is connected.
    \item $\Gamma$ has at least 5 vertices.
    \item $\Gamma$ is 3-point regular
    \item $\Gamma$ is formally self-dual.
\end{enumerate}
\end{customthm}
The reasons for requiring that $\Gamma$ be connected and have at least five vertices were a bit mysterious. The motivation for this paper was to understand more about the graphs excluded in \cite{Jae95} and to attempt to unify the classifications of spin models for knot polynomials.

To do this, we generalize slightly the notion of a knot polynomial. The result is a so-called singly-generated Yang-Baxter planar algebra, the classification of which can be found in \cite{BJ00}, \cite{BJ03}, \cite{BJL17}, and \cite{Liu15}. The main result of the latter paper is the following classification:
\begin{customthm}[\cite{Liu15}]
Let $\mathcal{V}$ be a singly-generated Yang-Baxter planar algebra. Then $\mathcal{V}$ is isomorphic to one of the following:
\begin{enumerate}[label = \roman{enumi}.]
    \item a specialization of the Jones polynomial, 
    \item a Bisch-Jones planar algebra,
    \item a specialization of the Kauffman polynomial, or
    \item a new family of planar algebras, which we call the Liu family. 
\end{enumerate}
\end{customthm}
Instead of looking at the matrix of weights of the crossing, we now consider the matrix of weights of a generating minimal projection, \begin{tikzpicture}[baseline = -.1cm, scale = .6]
\draw[fill = black!40!, draw = none] (-.35,-.5) rectangle (.35,.5);
\draw (-.35,-.5) -- (-.35,.5);
\draw (.35,-.5) -- (.35,.5);
\draw[fill = white] (-.5,-.3) rectangle (.5,.3);
\node at (0,0) {\tiny{{$P$}}};
\end{tikzpicture}, of a singly-generated YPBA. Moreover, because we have generalized the notion of a knot polynomial, we can now have symmetric and non-symmetric versions of spin models. The classification of symmetric spin models can be summarized as follows: 
\begin{customthm}
Let $\mathcal{V}$ be a singly-generated Yang-Baxter planar algebra with generator $\begin{tikzpicture}[baseline = -.1cm, scale = .6]
\draw[fill = black!40!, draw = none] (-.35,-.5) rectangle (.35,.5);
\draw (-.35,-.5) -- (-.35,.5);
\draw (.35,-.5) -- (.35,.5);
\draw[fill = white] (-.5,-.3) rectangle (.5,.3);
\node at (0,0) {\tiny{{$P$}}};
\end{tikzpicture}.$ Then $\mathcal{V}$ has a symmetric spin model if and only if the matrix of weights of \begin{tikzpicture}[baseline = -.1cm, scale = .6]
\draw[fill = black!40!, draw = none] (-.35,-.5) rectangle (.35,.5);
\draw (-.35,-.5) -- (-.35,.5);
\draw (.35,-.5) -- (.35,.5);
\draw[fill = white] (-.5,-.3) rectangle (.5,.3);
\node at (0,0) {\tiny{{$P$}}};
\end{tikzpicture} is the adjacency matrix of a graph $\Gamma$, and $\Gamma$ is (up to complementation) one of the following:
\begin{enumerate}[label = \roman{enumi}.]
\item the pentagon,
\item a disjoint union of complete graphs,
\item or a 3-point regular graph with $q_3 - 3q_2 + 3q_1 - q_0 \ne 0$, where the $q_i$ are parameters of the graph. 
\end{enumerate}
Moreover, $\mathcal{V}$ is a specialization of the Jones polynomial when $\Gamma$ is a complete graph, a Bisch-Jones planar algebra when it is a disjoint union of at least two complete graphs, and a specialization of the Kauffman polynomial otherwise.
\end{customthm}

Let $\Gamma$ be a 3-point regular graph with $n$ vertices. Then the computing time required to verify whether $\Gamma$ is formally self-dual is at least quadratic in $n$, as one must find the multiplicity of all eigenvalues of the adjacency matrix of $\Gamma$. Since the $q_i$ formula above is based on the parameters of the graph, however, it can be checked in constant time. Thus, the latter method is computationally better than the previous requirement. 

In addition to the symmetric spin models, we prove that there is a unique non-symmetric spin model for singly-generated Yang-Baxter planar algebras:

\begin{customthm}
Let $\mathcal{V}$ be a singly-generated Yang-Baxter planar algebra. $\mathcal{V}$ has a non-symmetric spin model if and only if the matrix of weights of \begin{tikzpicture}[baseline = -.1cm, scale = .6]
\draw[fill = black!40!, draw = none] (-.35,-.5) rectangle (.35,.5);
\draw (-.35,-.5) -- (-.35,.5);
\draw (.35,-.5) -- (.35,.5);
\draw[fill = white] (-.5,-.3) rectangle (.5,.3);
\node at (0,0) {\tiny{{$P$}}};
\end{tikzpicture} is the adjacency matrix of the 3-cycle. In this case, $\mathcal{V}$ is isomorphic to the unique Bisch-Jones planar algebra with $\dim V_3 = 9$.
\end{customthm}

 By considering all singly-generated YBPAs, we were able to give a shorter and conceptually clearer proof of the classifications found in \cite{Kau90}, \cite{Jae95}, and \cite{Jae95b}. Moreover, we discovered a new two-parameter family of spin models for specializations of Bisch-Jones planar algebras.
 
 This paper is organized into three sections: background, classification of symmetric spin models, and classification of non-symmetric spin models. The former section begins with defining some standard terms about undirected and directed graphs, which will be used in our classification for symmetric and non-symmetric spin models, respectively. Next, it will briefly define some concepts from planar algebras and will end with a formal definition of a spin model for a planar algebra. 

\subsection{Source Code}
All diagrams appearing in this paper can be found in the ArXiv source code, under the ``diagrams" subdirectory. 

\subsection{Acknowledgements}
The author would like to thank Noah Snyder for his support and advice throughout the writing of this paper. In addition, he would like to thank Dietmar Bisch, Vaughn Jones, and Zhengwei Liu for their results which undergird the classification in this paper. The author was also supported through NSF grant DMS-1454767.
\section{Background}

\subsection{Graph Theory}\label{graph}

The following definitions are standard from graph theory:

\begin{definition}
Let $\Gamma$ be a graph with $n$ vertices. The graph complement of $\Gamma,$ denoted $\Gamma^c$, is a graph of $n$ vertices such that any two distinct vertices in $\Gamma^c$ share an edge if and only if they do not share an edge in $\Gamma.$  
\end{definition}

\begin{definition}
Let $\Gamma$ be a graph and $W$ be a subset of the vertices of $\Gamma.$ The induced subgraph of $W$ is the subgraph of $\Gamma$ with vertex set $W$ and edge set consisting of the edges of $\Gamma$ with both endpoints in $W$.
\end{definition}

\begin{definition}\label{def:graphtypes}
Up to rotation, the following are the only possible subgraphs of distinct points $\{a,b,c\}$ for any graph $\Gamma.$
\begin{center}
\includegraphics[valign = c]{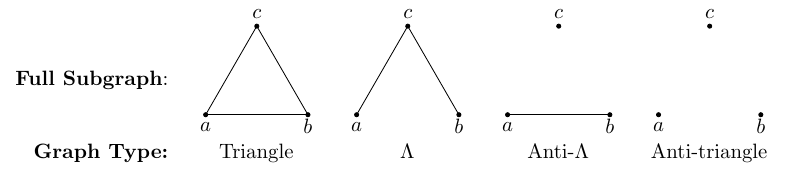}
\end{center}
A graph is triangle-free (resp., $\Lambda$-free, anti-$\Lambda$-free, or anti-triangle-free) if it has no induced subgraphs of the first (resp., second, third, fourth) type. 
\end{definition}

\begin{lemma}\label{lem:free}
A graph is is triangle-free if and only if its complement is anti-triangle-free. In addition, a graph is $\Lambda$-free if and only if its complement if anti-$\Lambda$-free.
\end{lemma}

\begin{proof}
These statements are obvious from the definitions.
\end{proof}

\begin{definition}\label{def:bipartite}
The complete bipartite graph $K_{m,n}$ is a graph composed of $n$ even vertices and $m$ odd vertices such that two vertices are adjacent if and only if they have opposite parity. In particular $\ast_n := K_{1,n}$ is the graph with one odd vertex, which we call $\star$, and a set $S$ of $n$ even vertices. 
\end{definition}

\begin{definition}\label{def:1regular}
A graph is said to be regular if every vertex has the same valence  (i.e. each vertex has the same number of neighbors).
\end{definition}

\begin{definition}\label{def:strong}
A graph $\Gamma$ is strongly regular if given any two (not necessarily distinct) vertices $a$ and $b$ of $\Gamma,$ the number of common neighbors they share depends only on whether $a = b$, $a$ is adjacent to $b$, or $a$ is not adjacent to $b$. These parameters are denoted, $k$, $\lambda$, and $\mu,$ respectively. For shorthand, $\Gamma$ will often be described as a $\text{srg}(n,k,\lambda,\mu),$ where $n$ is the number of vertices.
\end{definition}

We can also think of a strongly regular graph as having a fixed number of induced subgraphs of the following types depending on how $a$ and $b$ relate to each other:
\begin{center}
    \includegraphics[valign = c]{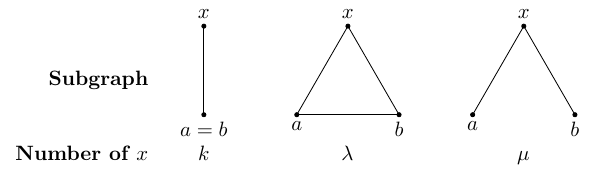}
\end{center}
The first diagram implies that a strongly regular graph is also regular. For a fixed $a$ and $b$ that are adjacent, a strongly regular graph has $\lambda$ triangles containing both vertices; for non-adjacent $a$ and $b$, there are $\mu$ $\Lambda$s containing both vertices.

\begin{lemma}\label{lem:cycle}
Let $\Gamma$ be a strongly regular graph with $k = 2$.  Then $\Gamma$ is isomorphic to either the pentagon, the square, or a disjoint union of triangles. 
\end{lemma}

\begin{proof}
Suppose that $\Gamma$ is a strongly regular graph with $k = 2$. Consider a connected component of $\Gamma$ with $m$ vertices $\{v_1,\dots,v_m\}$. Since $k=2$, every vertex has two neighbors so $m \geq 3$. Without loss of generality, suppose that $v_1$ is adjacent to $v_2$ and that $v_2$ is adjacent to $v_3.$ Since this graph is connected and regular with $k = 2$, $v_3$ is adjacent to $v_1$ if and only if $m = 3$. Now assume without loss of generality that $v_3$ is connected to $v_4$. Continuing this process, we see that $v_i$ is connected to $v_1$ if and only if $i = m.$ Thus, every connected component of $\Gamma$ is isomorphic to the $n$-gon, also called $C_n$. 

Consider $C_n$. Note that the triangle ($\text{srg}(3,2,1,0)$), square ($\text{srg}(4,2,0,2)$), and pentagon ($\text{srg}(5,2,0,1)$) are all strongly regular. Suppose $n\geq 6$. Let $\{v_1,\dots,v_n\}$ be the vertices of $C_n$ such that $v_i$ is adjacent to $v_{i+1}$ and $v_n$ is adjacent to $v_1.$ Thus, $v_1$ is not adjacent to $v_3$ or $v_4$. Since $n\geq 6$, though, we see that $v_1$ and $v_3$ have a common neighbor, $v_2$, while $v_1$ and $v_4$ have no common neighbors. Thus, $C_n$ is not strongly regular, and so each connected component of $\Gamma$ must be a triangle, square, or pentagon. 

Suppose that $\Gamma$ has at least two connected components. Thus, there exist two non-adjacent points which share no common neighbors. Since $\Gamma$ is strongly regular, it must be true that any two non-adjacent points share no common neighbors. Note that any two non-adjacent points on a square or pentagon have a common neighbor. A disjoint union of $m$ triangles, on the other hand, is still strongly regular with parameters $\text{srg}(3m, 2, 1, 0).$ Thus, if $\Gamma$ has at least two connected components it must be a disjoint union of triangles. In conclusion, if $\Gamma$ is a strongly regular graph with $k = 2$, then $\Gamma$ is the pentagon, the square, or a disjoint union of triangles, which completes our proof.  
\end{proof}

Strongly regular graphs have been extensively studied and much is known about them. The graphs that appear in this classification, though, satisfy a stronger condition than strong regularity, which we define presently in the style of \cite{Kup97}:

\begin{definition}\label{def:nptreg}
A graph is $n$-point regular if the number of common neighbors of any $n$ points depends only on how those $n$ points relate to each other.
\end{definition}

It is easy to see that 1-point regular graphs are just regular graphs and 2-point regular graphs are strongly regular graphs.

\begin{example}[3-point regular graphs] \label{threeptreg}
A graph is 3-point regular if the number of vertices connected to any 3 vertices is solely determined by how those three vertices relate. Up to rotation, the only graphs possible for any three distinct points are
\begin{center}
\includegraphics[valign = c]{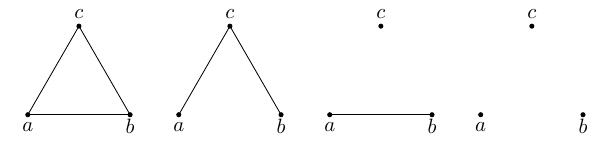}
\end{center}
We define the parameters of these graphs by

\begin{center}
\includegraphics[valign = c]{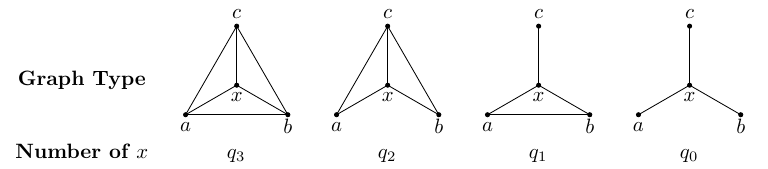}
\end{center}
These graphs have been studied extensively and partially classified. More information on them can be found in \cite{CGS78}.
\end{example}

\begin{proposition}\label{prop:lfree}
Let $\Gamma$ be a strongly regular, $\Lambda$-free graph with parameters $(n,k,\lambda,\mu)$. Then $\Gamma$ is isomorphic to $m K_{k+1},$ a disjoint union of $m = \dfrac{n}{k+1}$ complete graphs of size $k+1$.
\end{proposition}

\begin{proof}
Let $\Gamma$ be as described above. When $k = 0$, the graph has no edges, which can be thought of as a disjoint union of $m = \dfrac{n}{0+1} = n$ complete graphs of size $1$. When $k=1,$ the graph must be a disjoint union of $m = \dfrac{n}{2}$ complete graphs of size $2.$

Now suppose that $k \geq 2$. Let $x$ be a vertex of $\Gamma$ and $\{v_1,\dots,v_k\}$ be the vertices adjacent to $x$. Since $\Gamma$ is $\Lambda$-free, $v_i$ must also be adjacent to $v_j$ for all $i$ and $j$. Since each vertex only has $k$ neighbors, we see that the $v_i$ have no additional neighbors. Thus, the connected component containing $x$ is isomorphic to $K_{k+1}$. Since an connected component of $\Gamma$ must be isomorphic to $K_{k+1}$, $\Gamma$ must be a disjoint union of $m = \dfrac{n}{k+1}$ complete graphs of size $k+1$ when $k\geq 2$.
\end{proof}

\begin{corollary}\label{lfree}
Let $\Gamma$ be a strongly regular graph that is $\Lambda$-free. Then $\Gamma$ is 3-point regular. 
\end{corollary}

\begin{proof}
By Proposition \ref{prop:lfree}, If $\Gamma$ is $\Lambda$-free then $\Gamma$ is a disjoint union of $m$ complete graphs of size $k+1$. If $k \leq 2$ then the graph is vacuously 3-point regular. Suppose $k \geq 3$. Let $a$, $b$, and $c$ be any three vertices of $\Gamma$. Because $\Gamma$ is a disjoint union of complete graphs, there are no vertices adjacent to $a$, $b$, and $c$ if they are not pairwise adjacent. If $a$, $b$, and $c$ are all adjacent, there are $k-2$ vertices adjacent to all three. Thus, by definition $\Gamma$ is 3-point regular, as desired.
\end{proof}

\begin{lemma}\label{free}
Let $\Gamma$ be a strongly regular graph with parameters $(n,k,\lambda, \mu)$. Then up to complementation every graph can be classified as
\begin{enumerate}[label = \roman{enumi}.]
\item a disjoint union of complete graphs;
\item triangle-free, and neither $\Lambda$-free nor anti-$\Lambda$-free; or
\item neither triangle-free, $\Lambda$-free, anti-$\Lambda$-free, nor anti-triangle-free. 
\end{enumerate}
\end{lemma}

\begin{proof}
Let $\Gamma$ be as above. If $\Gamma$ or $\Gamma^c$ is $\Lambda$-free then by Proposition \ref{prop:lfree}, $\Gamma$ or its complement is a disjoint union of $m$ complete graphs of size $k+1$. Suppose that neither $\Gamma$ nor $\Gamma^c$ is $\Lambda$-free. By Lemma \ref{lem:free}, we can also assume that neither $\Gamma$ nor $\Gamma^c$ is anti-$\Lambda$-free. If $\Gamma$ or its complement is triangle-free and neither $\Lambda$-free nor anti-$\Lambda$-free, then it falls in the second category. If not, then we can also assume neither $\Gamma$ nor its complement is anti-triangle-free by Lemma \ref{lem:free}, which completes our proof.
\end{proof}

\subsubsection{Directed graphs}

The classification of non-symmetric spin models deals with directed graphs. We define some well-known terms below: 

\begin{definition}
A directed graph is composed of a finite number of vertices and edges. In addition, the edges have decorations of the following forms:
\begin{center}
\includegraphics[valign = c]{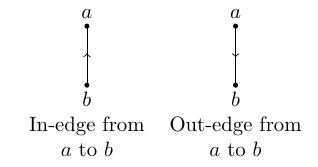}
\end{center}
When an edge goes from $a$ to $b$ (i.e. an out-edge from $a$ to $b$), we say that $a$ is adjacent to $b$.
\end{definition}

\begin{definition}\label{def:tournament}
A tournament is a directed graph where for all distinct $a$ and $b$ exactly one of the following holds: $a$ is adjacent to $b$ or $b$ is adjacent to $a$.
\end{definition}
The name ``tournament" was given to this graph because it represents the outcomes of a round-robin tournament (i.e. a tournament in which every pair of players play exactly one match). If player $a$ beats player $b$, then there is an arrow from $b$ to $a$. Since every player plays exactly once and there are no ties, the directed graph that appears is exactly the one described above.  

\begin{definition}\label{def:direg}
A directed graph is regular if every vertex has $k$ in-edges and $k$ out-edges.
\end{definition}

\begin{lemma}\label{lem:di}
Let $\Gamma$ be a tournament. A vertex in a regular tournament has $k = \dfrac{n-1}{2}$ in-edges and out-edges. If $k\geq 2$, the following subgraphs appear in $\Gamma$:
\begin{center}
    \includegraphics[valign = c]{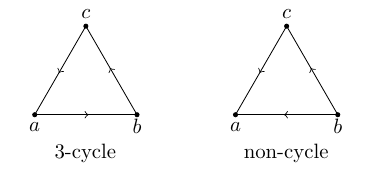}
\end{center}
for some $a,$ $b,$ and $c.$
\end{lemma}

\begin{proof}
Since $\Gamma$ is a tournament, the sum of the in-edges and out-edges for any vertex must be $n-1.$ If $\Gamma$ is regular,the number of in-edges and out-edges must be the same. Thus, $k = \dfrac{n-1}{2}.$ Suppose now that $k \geq 2.$ Then each edge has at least two in-edges and out-edges. Let $a$ be such a vertex and let $a$ be adjacent to $b$ and $c.$ Because $\Gamma$ is a tournament, either $b$ is adjacent to $c$ or $c$ is adjacent to $b$. Thus up to permuting $b$ and $c$, the non-cycle appears as a subgraph of $\Gamma$. Moreover, if $a$ has an out-edge to $b$, then by the pigeon-hole principle $b$ must have an out-edge to at least one of the in-edges of $a$. Thus, the 3-cycle also appears as a subgraph of $\Gamma,$ which completes our proof.  
\end{proof}

\subsection{Yang-Baxter planar algebras}

The following definitions are were given in \cite{Jon99} (See also \cite{Pet09}):

\begin{definition}
A shaded planar algebra is a collection of vector spaces $\mathcal{V} = \{V_{\pm i}\},$ $i \in \mathbb{Z}_{\geq 0}$ together with an action of the planar operad. 
\end{definition}

The planar operad is the collection of planar tangles, so-called ``spaghetti-and-meatball" diagrams, taken up to smooth isotopy. As an example, consider 
\begin{center}
    T~:=~\includegraphics[valign = c]{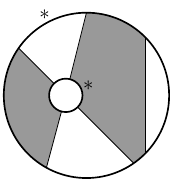}
\end{center}
In this case, this planar tangle will give a linear map  $T:V_{-2} \to V_{+3}.$ An action of the planar operad is an assignment of multilinear maps to planar tangles that respects the composition of tangles. 
\begin{definition}
A planar algebra, $\mathcal{V}$, is generated by a set of elements, $\mathcal{R},$ if for every element $v \in \mathcal{V},$ there exists a planar diagram, $T$, such that $T(r_1,\dots,r_k) = v$ for some $r_i \in \mathcal{R}.$ 
\end{definition}
\begin{definition}\label{def:std}
A planar diagram $T$ is drawn in \emph{standard form} if the input and output circles are drawn as rectangles, with strings attached only to the top or bottom such that the starred region represented by the left-hand side of the rectangle.
\end{definition}
In what follows, all planar diagrams are assumed to be in standard form unless explicitly drawn otherwise. In general, one can think of the planar operad as the set of operations on a planar algebra. Some of these operations appear often and have common names, shown below for two elements of a planar algebra $X$ and $Y$: 
\begin{center}
    \includegraphics[valign = c]{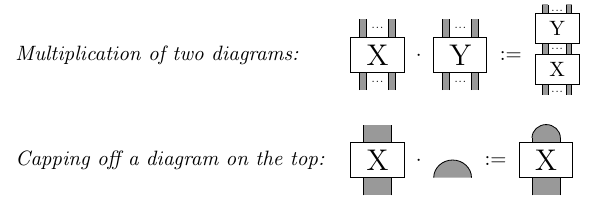}
\end{center}
\begin{definition}
A planar algebra is \emph{evaluable} if $\dim V_{+0} = \dim V_{-0} = 1.$ Let $\mathcal{V}$ be an evaluable planar algebra. We define an \emph{inner product} on $\mathcal{V}$ as
\begin{center}
\includegraphics[valign = c]{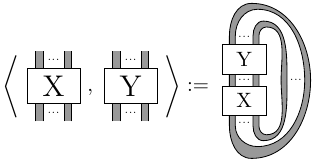}
\end{center}
and similarly for diagrams with the opposite shading. An element $v_n\in V_n$ that has the property that $<\!v_n,w_n\!> ~= 0$ for all $w_n\in V_n$ is called a \emph{negligible} element. We say that a planar algebra is \emph{non-degenerate} if it has no negligible elements.
\end{definition}

\begin{definition}
Let $\mathcal{V}$ and $\mathcal{W}$ be planar algebras. A planar algebra map, $\Phi:\mathcal{V} \to \mathcal{W},$ is a collection of linear maps $\phi_{\pm i}: V_{\pm i} \to W_{\pm i}$ intertwining the respective actions of the planar operad.
\end{definition}

In this paper we will classify certain maps between two special planar algebras. The injectivity of these maps will be important for many of the proofs that appear, but checking whether a planar algebra map has this property is easy in certain instances.

\begin{proposition}\label{inj}
Let $\mathcal{V}$ and $\mathcal{W}$ be planar algebras and $\Phi: \mathcal{V} \to \mathcal{W}$ be a planar algebra map. Suppose that $\mathcal{V}$ is evaluable, non-degenerate, and generated by a non-zero element $P$. Then $\Phi$ is injective if and only if $\Phi(P) \ne 0$. 
\end{proposition}%

\begin{proof}
Let $\mathcal{V}$ and $\mathcal{W}$ be as above. Since $\mathcal{V}$ is generated by $P$, any map $\Phi$ is completely described by its value, $\Phi(P)$.  If $\Phi$ is an injective map and $P$ is a non-zero element of $\mathcal{V}$, $\Phi(P)\ne 0$ by definition. Suppose that $\Phi(P) \ne 0$. Since $\mathcal{V}$ is non-degenerate and non-zero, any element of the kernel of $\Phi$ must be negligible. Thus, $\Phi$ is injective, as desired.  
\end{proof}

Examples of evaluable planar algebras are the Temperley-Lieb-Jones (TLJ) planar algebra \cite{Jon85} or the Kauffman polynomial planar algebra \cite{Kau90}. The Kauffman polynomial planar algebra is given below as an example, which is traditionally described as an \emph{unshaded} planar algebra. We can think of it as a shaded planar algebra by assuming the relations are the same in both possible shadings.

\begin{example}[The Kauffman polynomial planar algebra]\label{dubkau}
The Kauffman polynomial planar algebra is generated by \begin{tikzpicture}[baseline = -.15cm, scale = .65]
\begin{knot}[clip width = 4]
	\strand (-.5,.5) -- (.5,-.5);
	\strand (-.5,-.5) -- (.5,.5);
\end{knot}
\end{tikzpicture} and subject to the following relations:
\begin{center}
\begin{tabular}{ccc}
\multicolumn{3}{c}{$\begin{tikzpicture}[baseline]
\begin{knot}[clip width = 4]
	\strand (-.5,.5) -- (.5,-.5);
	\strand (-.5,-.5) -- (.5,.5);
\end{knot}
\end{tikzpicture}
 ~+~ \begin{tikzpicture}[baseline]
\begin{knot}[clip width = 4]
	\strand (-.5,-.5) -- (.5,.5);
	\strand (-.5,.5) -- (.5,-.5);
\end{knot}
\end{tikzpicture} = z \left( ~ \begin{tikzpicture}[baseline, rounded corners = 5mm]
    \draw (.5,.5) -- (.1,0) -- (.5,-.5);
    \draw (-.5,.5) -- (-.1,0) -- (-.5,-.5);
\end{tikzpicture} + \begin{tikzpicture}[baseline, rounded corners = 5mm]
    \draw (-.5,-.5) -- (0,.1) -- (.5,-.5);
    \draw (-.5,.5) -- (0,-.1) -- (.5,.5);
\end{tikzpicture} ~ \right)$}
\\
\\
$\skeincircle = d$	&& $\begin{tikzpicture}[baseline=.5cm, rounded corners = 5mm]
\draw (0,1)--(.4,.6);
\draw (.6,.4)--(1,0)--(0,0)--(1,1);
\end{tikzpicture} = a\cdot\begin{tikzpicture}[baseline = .5 cm, rounded corners = 6mm]
    \draw (0,1)--(.5,-.1) -- (1,1);
\end{tikzpicture}$ 
\\
\\
$\begin{tikzpicture}[baseline=(current bounding box.center), rounded corners = 3mm]
    \draw (0,0) -- (.75,.5) -- (0,1);
    \draw (.75,0) -- (.45,.2);
    \draw (.25,.3) -- (0,.5) -- (.3,.7);
    \draw (.45,.8)--(.75,1);
\end{tikzpicture}
 ~=~ \begin{tikzpicture}[baseline, rounded corners = 5mm]
    \draw (.5,.5) -- (.1,0) -- (.5,-.5);
    \draw (-.5,.5) -- (-.1,0) -- (-.5,-.5);
\end{tikzpicture}$ && $\begin{tikzpicture}[scale = .5, baseline]
\draw (1,1) -- (-1,-1);
\draw (1,-1) -- (.2,-.2);
\draw (-.2,.2) -- (-1,1);
\draw (0,1) to[out = -45, in = 135] (.4,.65);
\draw (.6,.45) to[out = -45, in = 90] (.8,0)
               to[out = -90, in = 45] (.6,-.45);
\draw (.4,-.65) to[out = 45, in = -135] (0,-1);       
\end{tikzpicture}
~=~
\begin{tikzpicture}[scale = .5, baseline]
\draw (1,1) -- (-1,-1);
\draw (1,-1) -- (.2,-.2);
\draw (-.2,.2) -- (-1,1);
\draw (0,1) to[out = 45, in = -135] (-.4,.65);
\draw (-.6,.45) to[out = -135, in = 90] (-.8,0)
                to[out = -90, in = 135] (-.6,-.45);
\draw (-.4,-.65) to[out = -45, in = 135] (0,-1);   
\end{tikzpicture}$
\end{tabular}
\end{center}
where $z(1+d) = a + a^{-1}.$ The last two relations are called Reidemeister Type II and III relations, respectively. Although we can think of this as a shaded planar algebra by imposing the above relations on both shadings, we instead impose the following lopsided relations for convenience (See \cite{MP14} for more on this.). These relations are the same as the usual definition after renormalizing. In addition to both shaded versions of the Reidemeister Type II and III relations, we have the following relations:
\begin{center}
    \includegraphics[valign = c, scale = .95]{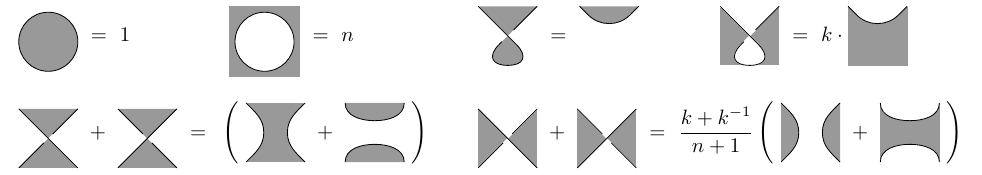}
\end{center}
\end{example}

By abstracting the notion of the Reidemeister relations, we can define a new class of planar algebras, called Yang-Baxter planar algebras, whose definition first appeared \cite{Liu15}. Examples of such planar algebras can also be found in \cite{BJ00}, \cite{BJ03}, and \cite{BJL17}. Unlike in those papers, however, we again use lopsided circle parameters for convenience.

\begin{definition}\label{def:YB}
A Yang-Baxter planar algebra is a non-degenerate, evaluable planar algebra with 
\begin{center}
\begin{tabular}{cccc}
$\begin{tikzpicture}[baseline]
\draw[fill = black!40!] circle (.5cm);
\end{tikzpicture} ~=~ 1$ &&& $\begin{tikzpicture}[baseline]
\draw[fill = black!40!, draw = none] (-.6,-.6) rectangle (.6,.6);
\draw[fill = white] circle (.5cm);
\end{tikzpicture} ~=~ n$
\end{tabular}
\end{center}
which is generated by a subset of $\mathcal{R}\subset V_{2+}$ satisfying the following relations:
\begin{center}
\begin{tabular}{ccc}
\includegraphics[valign = c, scale = 1]{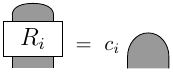} &&
\includegraphics[valign = c, scale = 1]{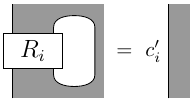} 
\\ Relation 1a && Relation 1b
\\
\includegraphics[valign = c, scale = 1]{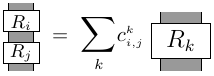}&&
\includegraphics[valign = c, scale = 1]{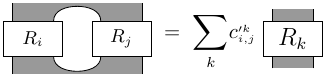}
\\ Relation 2a && Relation 2b
\\
$\displaystyle
    \includegraphics[valign = c, scale = .7]{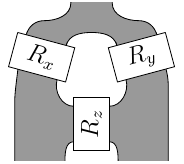} ~=~ \mathlarger{\mathlarger{\sum}}_{i,j,k}~c^{i,j,k}_{x,y,z} 
    \includegraphics[valign = c, scale = .7]{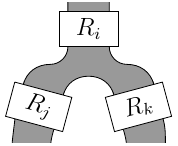}$
    &&
$\displaystyle
       \includegraphics[valign = c, scale = .7]{diagrams/pdf/shadedmiddle.pdf} ~=~ \mathlarger{\mathlarger{\sum}}_{x,y,z}~c^{x,y,z}_{i,j,k} \includegraphics[valign = c, scale = .7]{diagrams/pdf/unshadedmiddle.pdf}$
\\ Relation 3a && Relation 3b
\end{tabular}
\end{center}
for all $R_i,R_j,R_k\in\mathcal{R}\cup \left\{\begin{tikzpicture}[baseline = -.1cm, scale = .6]
    \draw[fill = black!40!, draw = none] (-.475,-.5) rectangle (.475,.5);
    \filldraw[fill = white] (.5,.5) to[out = -135, in = 90] (.2,0) to[out = -90, in = 135]  (.5,-.5);
    \filldraw[fill = white] (-.5,.5) to[out =-45, in = 90] (-.2,0) to[out = -90, in = 45] (-.5,-.5);
\end{tikzpicture}, \begin{tikzpicture}[baseline = -.1cm, scale = .6]
    \filldraw[fill = black!40!] (-.5,-.5) to[out = 90, in = 180] (0,-.2) to[out = 0, in = 90] (.5,-.5);
    \filldraw[fill = black!40!] (-.5,.5) to[out = -90, in = 180] (0,.2)  to[out = 0, in = -90] (.5,.5);
\end{tikzpicture}\right\}.$
\end{definition}

One can think of Relation 1a and 1b as an abstraction of the twist relations while Relations 2a and 2b and Relations 3a and 3b are generalizations of the Reidemeister II and III relations, respectively. 
\begin{definition}
A Yang-Baxter planar algebra is singly-generated if $\mathcal{R} = \left\{ \xdiagscaled[R] \right\}$ with $ \xdiagscaled[R] \ne 0 $ and $\left\{\xdiagscaled[R], \begin{tikzpicture}[baseline = -.1cm, scale = .6]
    \draw[fill = black!40!, draw = none] (-.475,-.5) rectangle (.475,.5);
    \filldraw[fill = white] (.5,.5) to[out = -135, in = 90] (.2,0) to[out = -90, in = 135]  (.5,-.5);
    \filldraw[fill = white] (-.5,.5) to[out =-45, in = 90] (-.2,0) to[out = -90, in = 45] (-.5,-.5);
\end{tikzpicture}, \begin{tikzpicture}[baseline = -.1cm, scale = .6]
    \filldraw[fill = black!40!] (-.5,-.5) to[out = 90, in = 180] (0,-.2) to[out = 0, in = 90] (.5,-.5);
    \filldraw[fill = black!40!] (-.5,.5) to[out = -90, in = 180] (0,.2)  to[out = 0, in = -90] (.5,.5);
\end{tikzpicture}\right\}$ spans the 2-box space. 
\end{definition}
The original definition of singly-generated in \cite{Liu15} requires that $\mathcal{R}\cup \left\{\begin{tikzpicture}[baseline = -.1cm, scale = .6]
    \draw[fill = black!40!, draw = none] (-.475,-.5) rectangle (.475,.5);
    \filldraw[fill = white] (.5,.5) to[out = -135, in = 90] (.2,0) to[out = -90, in = 135]  (.5,-.5);
    \filldraw[fill = white] (-.5,.5) to[out =-45, in = 90] (-.2,0) to[out = -90, in = 45] (-.5,-.5);
\end{tikzpicture}, \begin{tikzpicture}[baseline = -.1cm, scale = .6]
    \filldraw[fill = black!40!] (-.5,-.5) to[out = 90, in = 180] (0,-.2) to[out = 0, in = 90] (.5,-.5);
    \filldraw[fill = black!40!] (-.5,.5) to[out = -90, in = 180] (0,.2)  to[out = 0, in = -90] (.5,.5);
\end{tikzpicture}\right\}$ be a basis for the 2-box space. This requirement is softened so that we can include the case where \begin{tikzpicture}[baseline = -.1cm, scale = .6]
\draw[fill = black!40!, draw = none] (-.35,-.5) rectangle (.35,.5);
\draw (-.35,-.5) -- (-.35,.5);
\draw (.35,-.5) -- (.35,.5);
\draw[fill = white] (-.5,-.3) rectangle (.5,.3);
\node at (0,0) {\tiny{{$P$}}};
\end{tikzpicture} is a linear combination of the identity and cupcap. Suppose we have a singly-generated YBPA with generator $\begin{tikzpicture}[baseline = -.1cm, scale = .6]
\draw[fill = black!40!, draw = none] (-.35,-.5) rectangle (.35,.5);
\draw (-.35,-.5) -- (-.35,.5);
\draw (.35,-.5) -- (.35,.5);
\draw[fill = white] (-.5,-.3) rectangle (.5,.3);
\node at (0,0) {\tiny{{$P$}}};
\end{tikzpicture}.$ The fact that this diagram spans along with the identity and the cupcap implies the following additional relation, which we call ``Relation 0":
\begin{equation*}
\includegraphics[valign = c, scale = 1]{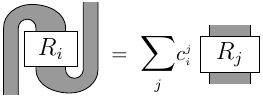}
\end{equation*}
\begin{lemma}\label{twodiag}
Let $\mathcal{V}$ be a singly-generated, evaluable, non-degenerate planar algebra. Let $X$ be the set of all diagrams in $\mathcal{V}$ of the form
\begin{center}
   \includegraphics[valign = c, scale = .7]{diagrams/pdf/unshadedmiddle.pdf} 
\end{center}
and $Y$ be the set of all diagrams of the form
\begin{center}
   \includegraphics[valign = c, scale = .7]{diagrams/pdf/shadedmiddle.pdf} 
\end{center}
where $R_i\in\{~\begin{tikzpicture}[baseline = -.1cm, scale = .6]
    \filldraw[fill = black!40!] (-.5,-.5) to[out = 90, in = 180] (0,-.2) to[out = 0, in = 90] (.5,-.5);
    \filldraw[fill = black!40!] (-.5,.5) to[out = -90, in = 180] (0,.2)  to[out = 0, in = -90] (.5,.5);
\end{tikzpicture},~\begin{tikzpicture}[baseline = -.1cm, scale = .6]
    \draw[fill = black!40!, draw = none] (-.475,-.5) rectangle (.475,.5);
    \filldraw[fill = white] (.5,.5) to[out = -135, in = 90] (.2,0) to[out = -90, in = 135]  (.5,-.5);
    \filldraw[fill = white] (-.5,.5) to[out =-45, in = 90] (-.2,0) to[out = -90, in = 45] (-.5,-.5);
\end{tikzpicture},~\begin{tikzpicture}[baseline = -.1cm, scale = .6]
\draw[fill = black!40!, draw = none] (-.35,-.5) rectangle (.35,.5);
\draw (-.35,-.5) -- (-.35,.5);
\draw (.35,-.5) -- (.35,.5);
\draw[fill = white] (-.5,-.3) rectangle (.5,.3);
\node at (0,0) {\tiny{{$P$}}};
\end{tikzpicture}~\}.$ Then up to a scalar
\[(X\cap Y)' = \left\{ ~\includegraphics[valign = c, scale = .7]{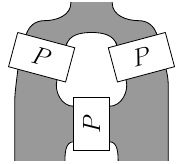} ~,~\includegraphics[valign = c, scale = .7]{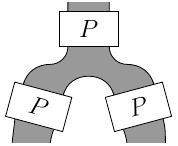}  \right\}\]
That is, every diagram but the above two appear in both sets. 
\end{lemma}

\begin{proof}
Let $X$ and $Y$ be as described above. Then after simplification we have that $X$ is composed of the following 15 diagrams:
\begin{center}
\includegraphics[scale = .6]{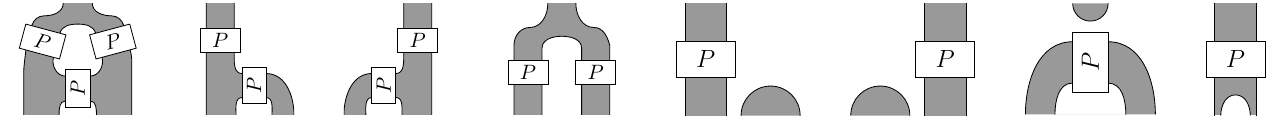}

\vspace{.15in}

\includegraphics[scale = .6]{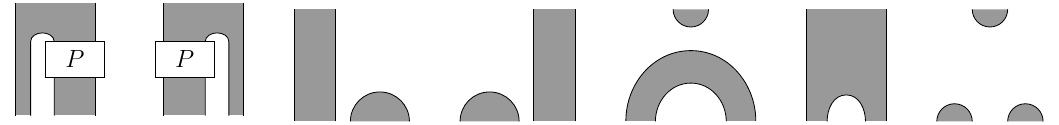}  
\end{center}
and $Y$ is the set of the following 15 diagrams after simplification:
\begin{center}
\includegraphics[scale = .6]{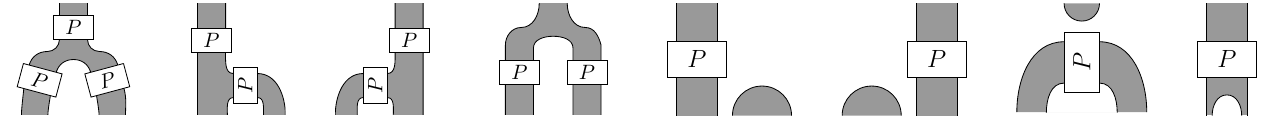}

\vspace{.15in}

\includegraphics[scale = .6]{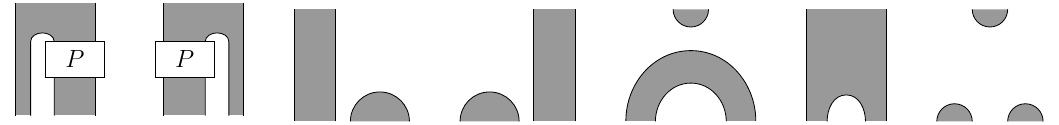}  
\end{center}
Thus, \[(X\cap Y)' = \left\{ ~\includegraphics[valign = c, scale = .6]{diagrams/pdf/unshadedmiddlep.pdf} ~,~\includegraphics[valign = c, scale = .6]{diagrams/pdf/shadedmiddlep.pdf}  \right\}\]

\noindent as desired. 
\end{proof}

\begin{proposition}\label{prop:onlyp}
Let $\mathcal{V}$ be an evaluable, non-degenerate planar algebra generated by $\begin{tikzpicture}[baseline = -.1cm, scale = .6]
\draw[fill = black!40!, draw = none] (-.35,-.5) rectangle (.35,.5);
\draw (-.35,-.5) -- (-.35,.5);
\draw (.35,-.5) -- (.35,.5);
\draw[fill = white] (-.5,-.3) rectangle (.5,.3);
\node at (0,0) {\tiny{{$P$}}};
\end{tikzpicture}$ with
\begin{center}
\begin{tabular}{cccc}
$\begin{tikzpicture}[baseline]
\draw[fill = black!40!] circle (.5cm);
\end{tikzpicture} ~=~ 1$ &&& $\begin{tikzpicture}[baseline]
\draw[fill = black!40!, draw = none] (-.6,-.6) rectangle (.6,.6);
\draw[fill = white] circle (.5cm);
\end{tikzpicture} ~=~ n$
\end{tabular}
\end{center}
Then $\mathcal{V}$ is a Yang-Baxter planar algebra if and only if the defining relations exist when $R_i = R_j = R_k = P.$ 
\end{proposition}

\begin{proof}
Clearly, if $\mathcal{V}$ is a YBPA then the defining relations exist when $R_i = R_j = R_k = P$ by definition. Suppose now that the defining relations exist when $R_i = R_j = R_k = P$. For any evaluable planar algebra, note that when $R_i \in \left\{\begin{tikzpicture}[baseline = -.1cm, scale = .6]
    \draw[fill = black!40!, draw = none] (-.475,-.5) rectangle (.475,.5);
    \filldraw[fill = white] (.5,.5) to[out = -135, in = 90] (.2,0) to[out = -90, in = 135]  (.5,-.5);
    \filldraw[fill = white] (-.5,.5) to[out =-45, in = 90] (-.2,0) to[out = -90, in = 45] (-.5,-.5);
\end{tikzpicture}, \begin{tikzpicture}[baseline = -.1cm, scale = .6]
    \filldraw[fill = black!40!] (-.5,-.5) to[out = 90, in = 180] (0,-.2) to[out = 0, in = 90] (.5,-.5);
    \filldraw[fill = black!40!] (-.5,.5) to[out = -90, in = 180] (0,.2)  to[out = 0, in = -90] (.5,.5);
\end{tikzpicture}\right\}$ Relations 1a and 1b exist and are completely determined by the circle parameter, $n$, since $\mathcal{V}$ is evaluable. Thus, Relations 1a and 1b exist for all $R_i \in \left\{\begin{tikzpicture}[baseline = -.1cm, scale = .6]
\draw[fill = black!40!, draw = none] (-.35,-.5) rectangle (.35,.5);
\draw (-.35,-.5) -- (-.35,.5);
\draw (.35,-.5) -- (.35,.5);
\draw[fill = white] (-.5,-.3) rectangle (.5,.3);
\node at (0,0) {\tiny{{$P$}}};
\end{tikzpicture}, \begin{tikzpicture}[baseline = -.1cm, scale = .6]
    \draw[fill = black!40!, draw = none] (-.475,-.5) rectangle (.475,.5);
    \filldraw[fill = white] (.5,.5) to[out = -135, in = 90] (.2,0) to[out = -90, in = 135]  (.5,-.5);
    \filldraw[fill = white] (-.5,.5) to[out =-45, in = 90] (-.2,0) to[out = -90, in = 45] (-.5,-.5);
\end{tikzpicture}, \begin{tikzpicture}[baseline = -.1cm, scale = .6]
    \filldraw[fill = black!40!] (-.5,-.5) to[out = 90, in = 180] (0,-.2) to[out = 0, in = 90] (.5,-.5);
    \filldraw[fill = black!40!] (-.5,.5) to[out = -90, in = 180] (0,.2)  to[out = 0, in = -90] (.5,.5);
\end{tikzpicture}\right\}$ for any $\mathcal{V}$ as defined above.  

For a singly-generated YBPA, Relations 2a and 2b are determined by the circle parameter and Relations 1a and 1b when at least one of $R_i$ or $R_j \in \left\{\begin{tikzpicture}[baseline = -.1cm, scale = .6]
    \draw[fill = black!40!, draw = none] (-.475,-.5) rectangle (.475,.5);
    \filldraw[fill = white] (.5,.5) to[out = -135, in = 90] (.2,0) to[out = -90, in = 135]  (.5,-.5);
    \filldraw[fill = white] (-.5,.5) to[out =-45, in = 90] (-.2,0) to[out = -90, in = 45] (-.5,-.5);
\end{tikzpicture}, \begin{tikzpicture}[baseline = -.1cm, scale = .6]
    \filldraw[fill = black!40!] (-.5,-.5) to[out = 90, in = 180] (0,-.2) to[out = 0, in = 90] (.5,-.5);
    \filldraw[fill = black!40!] (-.5,.5) to[out = -90, in = 180] (0,.2)  to[out = 0, in = -90] (.5,.5);
\end{tikzpicture}\right\}.$ Thus, by assumption there exist relations 2a and 2b for all $R_i$ and $R_j$. By Lemma \ref{twodiag}, when at least one of $R_i$, $R_j$ or $R_k\in\left\{\begin{tikzpicture}[baseline = -.1cm, scale = .6]
    \draw[fill = black!40!, draw = none] (-.475,-.5) rectangle (.475,.5);
    \filldraw[fill = white] (.5,.5) to[out = -135, in = 90] (.2,0) to[out = -90, in = 135]  (.5,-.5);
    \filldraw[fill = white] (-.5,.5) to[out =-45, in = 90] (-.2,0) to[out = -90, in = 45] (-.5,-.5);
\end{tikzpicture}, \begin{tikzpicture}[baseline = -.1cm, scale = .6]
    \filldraw[fill = black!40!] (-.5,-.5) to[out = 90, in = 180] (0,-.2) to[out = 0, in = 90] (.5,-.5);
    \filldraw[fill = black!40!] (-.5,.5) to[out = -90, in = 180] (0,.2)  to[out = 0, in = -90] (.5,.5);
\end{tikzpicture}\right\}$ Relations 3a and 3b are the identity relation. Thus, there exist relations 3a and 3b for all $R_i$, $R_j,$ and $R_k.$ Thus, $\mathcal{V}$ is a singly-generated YBPA by definition, which completes our proof. 
\end{proof}

In \cite{Liu15}, Liu completed the classification begun by Bisch and Jones in \cite{BJ00}, \cite{BJ03}, and \cite{BJL17}. This deep result is stated below:

\begin{theorem}[Theorem 1.3 in \cite{Liu15}]\label{Liu}
Let $\mathcal{V}$ be a singly-generated Yang-Baxter planar algebra. Then $\mathcal{V}$ is isomorphic to one of the following:
\begin{enumerate}[label = \roman{enumi}.]
    \item a Tempereley-Lieb-Jones (TLJ) planar algebra, 
    \item a Bisch-Jones planar algebra,
    \item a Kauffman polynomial planar algebra, or
    \item a planar algebra in the Liu family.
\end{enumerate}
\end{theorem}

The term ``Bisch-Jones planar algebra" refers to the planar algebras in \cite{BJ00} and \cite{BJ03}. They include a unique planar algebra with $\dim V_3 = 9$ and $10$, a one-parameter family of planar algebras with $\dim V_3 = 11$ and a two-parameter family of planar algebras with $\dim V_3 = 12.$ Each of these planar algebras will appear in our classification. The latter family in the above list, however, does not appear in our classification so its definition is omitted.

\subsection{Graph planar algebra of $\ast_n$.}
Another class of planar algebras is the planar algebra of a bipartite graph originally defined in \cite{Jon00}. In this paper, we will only make use of the planar algebra of a special bipartite graph, $\ast_n,$ described in Definition \ref{def:bipartite}, which is called the spin planar algebra in \cite{Jon99}. The following definitions were given in \cite{Jon99}:

\begin{definition}
Consider $\ast_n$. Let $\star$ be its odd vertex and $S$ be the set of $n$ even vertices. Define $\text{GPA}(\ast_n)_{\pm i}$ to be the set of linear functionals of based loops of length $2i$. Functionals in $V_{-i}$ have base point $\star$, while functionals in $V_{+i}$ have base point in $S.$    
\end{definition}
 
Because our graph is $\ast_n$, note that $\text{GPA}(\ast_n)_{\pm i} \cong \mathbb{C}[S^i],$ the vector space of formal $\mathbb{C}$-linear combinations of $i$-tuples of $S$. We would like for the collection of these vector spaces, $\text{GPA}(\ast_n),$ to be a planar algebra. To that end, we define the action of the planar operad. 
\begin{definition}
A state $\sigma$ on a planar diagram $T:V_{i_1}\otimes\dots\otimes V_{i_k}\to V_{o}$ is an assignment of elements of $S$ to the unshaded regions and $\star$ to shaded regions. Define $\sigma|_i$ to be the based loop of $\ast_n$ obtained by reading the inputs clockwise around the $i$-th input disc, starting with the starred region. Let $\sigma|_o$ be the based loop obtained by performing the same operation around the output disc. For linear functionals $f_1,\dots,f_k$, define the action of $\text{GPA}(\ast_n),$ on $T$ to be the functional
\begin{equation*}\label{eq:statesum}
T(f_1,f_2,\dots,f_k)(l) = 
\sum_{\substack{\sigma \text{ with}\\
   \sigma|_o = l}}
~ \prod_{i=1}^k f_i(\sigma|_i)
\end{equation*}
\end{definition}

It is important to note that the general definition of a state is different than the one given here, as we are making use of the vector space isomorphism mentioned above. In addition, we have a omitted the normalization factor found in the standard definition of \cite{Jon99} and \cite{Jon00} which has the effect that 
\begin{center}
\begin{tabular}{cccc}
$\begin{tikzpicture}[baseline]
\draw[fill = black!40!] circle (.5cm);
\end{tikzpicture} ~=~ 1$ &&& $\begin{tikzpicture}[baseline]
\draw[fill = black!40!, draw = none] (-.6,-.6) rectangle (.6,.6);
\draw[fill = white] circle (.5cm);
\end{tikzpicture} ~=~ n$
\end{tabular}
\end{center}
To better understand the action of the planar operad, we give an example using basis elements of $\text{GPA}(\ast_n),$ Kronecker deltas of loops, which we denote $\delta_l.$ As a shorthand, we will write $l$ as a concatenation of the vertices omitting $\star.$ Consider $\ast_5$ and say $S = \{a,b,c,d,e\}$. Then consider the planar tangle below:
\begin{center}
\includegraphics[valign = c]{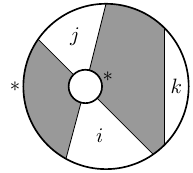}
\end{center}
where $i,j,k$ are elements of $S$ (The shaded regions are labelled with $\star$ but are omitted to avoid confusion with the starred regions.). Consider the loop $ab$. Then by definition of the action of the planar operad 
\begin{equation*}
    T(\delta_{ab})(\delta_{ijk}) = \left\{
    \begin{array}{cc}
        1 & i=a \text{ and } j=b  \\
        0 & \text{else} 
    \end{array}\right.
    ~ = ~\sum_{k\in S} \delta_{abk}
\end{equation*}
One can extend this result linearly to get a linear functional for any choice of input. By considering the action of  $\mathcal{V}$ on planar tangles with no input discs, we see that
\begin{center}
\begin{tabular}{ccccc}
     $\begin{tikzpicture}[baseline = -.1cm]
    \filldraw[fill = black!40!] (-.5,-.5) to[out = 90, in = 180] (0,-.2) to[out = 0, in = 90] (.5,-.5);
    \filldraw[fill = black!40!] (-.5,.5) to[out = -90, in = 180] (0,.2)  to[out = 0, in = -90] (.5,.5);
    \node at (-.5,0) {$a$};
    \node at (.5,0) {$b$};
\end{tikzpicture} ~=~ \left\{
    \begin{array}{cc}
        1 & a = b  \\
        0 & a \ne b 
    \end{array}
    \right.  $ && and &&
    \begin{tikzpicture}[baseline = -.1cm]
    \draw[fill = black!40!, draw = none] (-.485,-.5) rectangle (.485,.5);
    \filldraw[fill = white] (.5,.5) to[out = -135, in = 90] (.2,0) to[out = -90, in = 135]  (.5,-.5);
    \filldraw[fill = white] (-.5,.5) to[out =-45, in = 90] (-.2,0) to[out = -90, in = 45] (-.5,-.5);
    \node at (-.5,0) {$a$};
    \node at (.5,0) {$b$};
\end{tikzpicture} = 1
\end{tabular}
\end{center}
for all $a,b\in S.$ Thus, we know that 
\begin{equation}
  \begin{tikzpicture}[baseline = -.1cm]
    \filldraw[fill = black!40!] (-.5,-.5) to[out = 90, in = 180] (0,-.2) to[out = 0, in = 90] (.5,-.5);
    \filldraw[fill = black!40!] (-.5,.5) to[out = -90, in = 180] (0,.2)  to[out = 0, in = -90] (.5,.5);
\end{tikzpicture} ~=~ \sum_{i=1}^n \delta_{ii}
\end{equation}
and
\begin{equation}
  \begin{tikzpicture}[baseline = -.1cm]
    \draw[fill = black!40!, draw = none] (-.485,-.5) rectangle (.485,.5);
    \filldraw[fill = white] (.5,.5) to[out = -135, in = 90] (.2,0) to[out = -90, in = 135]  (.5,-.5);
    \filldraw[fill = white] (-.5,.5) to[out =-45, in = 90] (-.2,0) to[out = -90, in = 45] (-.5,-.5);
\end{tikzpicture} ~=~ \sum_{i,j=1}^n \delta_{ij}
\end{equation}
\label{YB}

\subsection{Spin Models}

Given the connection that knot polynomials have to spin models, it is logical that we can define a similar concept for planar algebras. The following definition is based on \cite{Jon99}:
\begin{definition}
Let $\mathcal{V}$ be a planar algebra. A spin model for $\mathcal{V}$ is a map of planar algebras $\Phi:\mathcal{V}\to\text{GPA}(\ast_n)$ for some $n$.
\end{definition}
Suppose $\mathcal{V}$ is generated by elements of $V_{2+}$. Recall that a basis for $\text{GPA}(\ast_n)_{2k}$ is the collection of Kronecker deltas of loops,  $\delta_{l}$, where $l$ is any $k$-tuple of elements of $S$. If $X$ is an element of $V_{+2}$ then, in general, $\Phi(X) =\displaystyle \sum_{a,b\in S} c_X(a,b)\cdot\delta_{ab},$ where $c_X(a,b) \in\mathbb{C}.$ Let $C_X$ be the $n\times n$ matrix where the $(a,b)$ entry of $C_X$ is $c_X(a,b).$ By the action of the planar operad on $\text{GPA}(\ast_n),$ we know that the matrix corresponding to \begin{tikzpicture}[baseline = -.1cm, scale = .6]
    \filldraw[fill = black!40!] (-.5,-.5) to[out = 90, in = 180] (0,-.2) to[out = 0, in = 90] (.5,-.5);
    \filldraw[fill = black!40!] (-.5,.5) to[out = -90, in = 180] (0,.2)  to[out = 0, in = -90] (.5,.5);
\end{tikzpicture} is the identity matrix; similarly, the matrix corresponding to \begin{tikzpicture}[baseline = -.1cm, scale = .6]
    \draw[fill = black!40!, draw = none] (-.475,-.5) rectangle (.475,.5);
    \filldraw[fill = white] (.5,.5) to[out = -135, in = 90] (.2,0) to[out = -90, in = 135]  (.5,-.5);
    \filldraw[fill = white] (-.5,.5) to[out =-45, in = 90] (-.2,0) to[out = -90, in = 45] (-.5,-.5);
\end{tikzpicture} is the matrix of all ones.  

To capture the fact that $\Phi$ is a map of planar algebras, we will write 
\begin{equation*}
c_X(a,b) ~:=~ \xdiag[X]  
\end{equation*}
for all $a,b\in S$ and generators $X$. More generally we can give the following definition:

\begin{definition}
Suppose $\Phi$ gives a spin model for some shaded planar algebra, $\mathcal{V}$, that is generated by 2-boxes and let $X\in V_{\pm k}$. Then a state sum with respect to $(a_1,\dots,a_k)$ is the sum over all possible labelings of the unshaded regions of $X$ such that the $i$th exterior region of $X$ (starting with the starred region and going clockwise) is $a_i.$  
\end{definition}
As an example, the state sum of 
\begin{center}
    \includegraphics[valign = c]{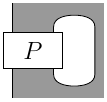}
\end{center}
with respect to a is
\begin{equation*}
\includegraphics[valign = c]{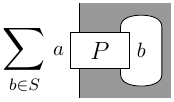}  
\end{equation*}
Just like for $c_X(a,b)$ for some generator $X$, one should think of the state sum of any diagram $Y$ with respect to $(a_1,\dots,a_i)$ as the coefficient of $\delta_{a_1\cdots a_i}$ of $\Phi(Y).$  Suppose $\mathcal{V}$ is a Yang-Baxter planar algebra with a spin model. Then any relation of $\mathcal{V}$ gives corresponding relations to the underlying state sums for any assignment of atoms to the exterior regions. Importantly, the converse is also true when $\mathcal{V}$ is a singly-generated YBPA.

\begin{proposition}\label{prop:atoms}
Let $X$ and $Y$ be elements of $V_{\pm k}$ of a singly-generated YBPA, $\mathcal{V}$, and let $\Phi$ be a spin model of $\mathcal{V}$. Let $P$ be its generator and suppose that $\Phi(P) \ne 0$. Then the state sum of $X$ and $Y$ with respect to $(a_1,\dots,a_k)$ are equal for all possible $a_i \in S$ if and only if $X = Y$. 
\end{proposition}

\begin{proof}
Clearly, if $X = Y$ then the state sum of $X$ and $Y$ with respect to $(a_1,\dots,a_k)$ are equal for all $a_i$. Suppose now that the state sum of $X$ and $Y$ with respect to $(a_1,\dots,a_k)$ are equal for all $a_i$. Then the state sum of $X - Y$ is 0 for all possible assignment of atoms to the exterior regions of $X - Y$. Thus, by definition of a spin model, $\Phi(X - Y) = 0$. Since $\Phi(P) \ne 0$, Proposition \ref{inj} tells us that $X - Y = 0,$ as desired. 
\end{proof}

For a spin model of a singly-generated YBPA, $\mathcal{V}$, Proposition \ref{prop:atoms} tells us that if the state sums of any two diagrams of the same box space are equal for all possible assignment of atoms, then the diagrams themselves are equal. In particular, if a relation holds for all possible state sums, then that relation must hold in $\mathcal{V}$.

\begin{definition}\label{def:sym}
A spin model is symmetric if $c_x(a,b) = c_x(b,a)$ for all $a,b\in S$. 
\end{definition}

In \cite{Kau87} and \cite{Jae95}, Kauffman and Jaeger noted that the matrix of weights for the crossing was equal to the adjacency matrix of some graph $\Gamma$.
These theorems can be stated as follows (See Section \ref{graph} for any graph theory definitions):

\begin{theorem}[\cite{Kau87}]
Let $\mathcal{V}$ be a specialization of the Tempereley-Lieb-Jones planar algebra. Let $\Phi:\mathcal{V} \to \text{GPA}(\ast_n)$ be a spin model and $C_X$ be the matrix of weights for the crossing. Then $\Phi$ exists if and only if $C_X$ is the adjacency matrix of the complete graph, $K_n$. 
\end{theorem}
\begin{theorem}[\cite{Jae95}]
Let $\mathcal{V}$ be a specialization of the Kauffman polynomial planar algebra. Let $\Phi:\mathcal{V} \to \text{GPA}(\ast_n)$ be a spin model and $C_X$ be the matrix of weights for the crossing. Then $\Phi$ exists if and only if $C_X$ is the adjacency matrix of some graph $\Gamma$ of $n$ vertices with the following properties:
\begin{enumerate}[label = \roman{enumi}.]
    \item $\Gamma$ has at least 5 vertices.
    \item $\Gamma$ and $\Gamma^c$ are connected.
    \item $\Gamma$ is 3-point regular.
    \item $\Gamma$ is formally self-dual.
\end{enumerate}
\end{theorem}
The requirement of formal self-duality is a technical condition on the graph that can be cumbersome to check. By generalizing this result to all singly-generated Yang-Baxter planar algebras, the formal self-duality condition will be reformulated into a more manageable formula in terms of the 3-point regular parameters of the graph. 

\begin{example}[Spin models for the Kauffman polynomial]
The following graphs are known to give spin models for the Kauffman polynomial:
\begin{center}
\begin{tabular}{r|cccccccc}
     \multirow{2}{*}{Graph Name} & \multicolumn{8}{c}{Parameters of the graph}  \\
                 & $n$ & $k$ & $\lambda$ & $\mu$ & $q_3$ & $q_2$ & $q_1$ & $q_0$ \\
     \hline
     Pentagon    & 5   & 2  & 0 & 1 & n/a & n/a & n/a & n/a \\
     9-Paley     & 9   & 4  & 1 & 2 & 0 & 0 & 1 & 0 \\
     Clebsch     & 16  & 5  & 0 & 2 & 0 & 0 & 0 & 1 \\
     Higman-Sims & 100 & 22 & 0 & 6 & 0 & 0 & 0 & 2 \\ 
\end{tabular}
\end{center}
Aside from these examples, it is still an open question whether there exist any other graphs that give spin models for the Kauffman polynomial. It is our hope that the classification here will shed some light on other graphs in this hypothetical family.
\end{example}

\begin{example}[A non-symmetric spin model]
In \cite{Jae95b}, it is shown that the 3-cycle gives a non-symmetric spin model for the unique Bisch-Jones planar algebra with $\dim V_3 = 9$ first defined in \cite{BJ00}. This planar algebra is also called the group planar algebra for $\mathbb{Z}/3$. More information on group planar algebras and subgroup-subfactor planar algebras generally can be found in \cite{Gup08}.    
\end{example}\label{spin}

\section{Classification of spin models for singly-generated YBPAs}

For the remainder of this paper, assume that $\mathcal{V}$ is a singly-generated Yang-Baxter planar algebra and let \begin{tikzpicture}[baseline = -.1cm, scale = .6]
\draw[fill = black!40!, draw = none] (-.35,-.5) rectangle (.35,.5);
\draw (-.35,-.5) -- (-.35,.5);
\draw (.35,-.5) -- (.35,.5);
\draw[fill = white] (-.5,-.3) rectangle (.5,.3);
\node at (0,0) {\tiny{{$P$}}};
\end{tikzpicture} be its generator. For brevity, we will use the shorthand $P$ for \begin{tikzpicture}[baseline = -.1cm, scale = .6]
\draw[fill = black!40!, draw = none] (-.35,-.5) rectangle (.35,.5);
\draw (-.35,-.5) -- (-.35,.5);
\draw (.35,-.5) -- (.35,.5);
\draw[fill = white] (-.5,-.3) rectangle (.5,.3);
\node at (0,0) {\tiny{{$P$}}};
\end{tikzpicture}. Since a spin model for $\mathcal{V}$ is a map of planar algebras, $\Phi:\mathcal{V}\to{GPA}(\ast_n)$, we need only define where the generator of $\mathcal{V}$ is sent and check that this assignment respects the relations of $\mathcal{V}$. Because we have the relation 
\begin{center}
$\begin{tikzpicture}[baseline]
\draw[fill = black!40!, draw = none] (-.6,-.6) rectangle (.6,.6);
\draw[fill = white] circle (.5cm);
\end{tikzpicture} ~=~ n$
\end{center}
this automatically implies that $|S| = n.$ Thus $n\in\mathbb{Z}_{> 0}$. We will refer to the other relations of Definition \ref{def:YB} by the names given to each in the definition (e.g. Relation 1a, etc.).

Because we have chosen our planar algebra to be non-degenerate, $V_{+2}$ forms a semi-simple algebra (with multiplication defined as vertical stacking). Moreover, since the $\dim V_{+2} = 2$ or $3$ we know that $V_2 \simeq \mathbb{C}^2$ or $\mathbb{C}^3$ as an algebra. Hence, without loss of generality, we may assume that the generator, $P,$ is a minimal idempotent of $V_{+2}$ orthogonal to \begin{tikzpicture}[baseline = -.1cm, scale = .6]
    \filldraw[fill = black!40!] (-.5,-.5) to[out = 90, in = 180] (0,-.2) to[out = 0, in = 90] (.5,-.5);
    \filldraw[fill = black!40!] (-.5,.5) to[out = -90, in = 180] (0,.2)  to[out = 0, in = -90] (.5,.5);
\end{tikzpicture}. 

\begin{remark}
It was explained in \cite{Jon99} (and further clarified in Jones' lectures \cite{Jon11}) that the projections in the 2-box space are closely related to association schemes and their Bose-Mesner algebras. In our setting these association schemes will all come from graphs.
\end{remark}

Let \begin{tikzpicture}[baseline = -.1cm, scale = .6]
\draw[fill = black!40!, draw = none] (-.35,-.5) rectangle (.35,.5);
\draw (-.35,-.5) -- (-.35,.5);
\draw (.35,-.5) -- (.35,.5);
\draw[fill = white] (-.5,-.3) rectangle (.5,.3);
\node at (0,0) {\tiny{{$Q$}}};
\end{tikzpicture}, or $Q$, be the other minimal idempotent of $V_{+2}$ such that $Q$ is orthogonal to both $P$ and $\begin{tikzpicture}[baseline = -.1cm, scale = .6]
    \filldraw[fill = black!40!] (-.5,-.5) to[out = 90, in = 180] (0,-.2) to[out = 0, in = 90] (.5,-.5);
    \filldraw[fill = black!40!] (-.5,.5) to[out = -90, in = 180] (0,.2)  to[out = 0, in = -90] (.5,.5);
\end{tikzpicture}$. By direct computation, we see that 
\begin{equation}\label{eq:qrela}
 \begin{tikzpicture}[baseline = -.1cm]
\draw[fill = black!40!, draw = none] (-.35,-.5) rectangle (.35,.5);
\draw (-.35,-.5) -- (-.35,.5);
\draw (.35,-.5) -- (.35,.5);
\draw[fill = white] (-.5,-.3) rectangle (.5,.3);
\node at (0,0) {\large{{$Q$}}};
\end{tikzpicture}   ~=~ \begin{tikzpicture}[baseline = -.1cm]
    \draw[fill = black!40!, draw = none] (-.485,-.5) rectangle (.485,.5);
    \filldraw[fill = white] (.5,.5) to[out = -135, in = 90] (.2,0) to[out = -90, in = 135]  (.5,-.5);
    \filldraw[fill = white] (-.5,.5) to[out =-45, in = 90] (-.2,0) to[out = -90, in = 45] (-.5,-.5);
\end{tikzpicture} ~-~ \begin{tikzpicture}[baseline = -.1cm]
    \filldraw[fill = black!40!] (-.5,-.5) to[out = 90, in = 180] (0,-.2) to[out = 0, in = 90] (.5,-.5);
    \filldraw[fill = black!40!] (-.5,.5) to[out = -90, in = 180] (0,.2)  to[out = 0, in = -90] (.5,.5);
\end{tikzpicture} ~-~ \begin{tikzpicture}[baseline = -.1cm]
\draw[fill = black!40!, draw = none] (-.35,-.5) rectangle (.35,.5);
\draw (-.35,-.5) -- (-.35,.5);
\draw (.35,-.5) -- (.35,.5);
\draw[fill = white] (-.5,-.3) rectangle (.5,.3);
\node at (0,0) {\large{{$P$}}};
\end{tikzpicture} 
\end{equation}
Thus, 
\begin{equation}\label{eq:qrel}
 \begin{tikzpicture}[baseline = -.1cm]
    \draw[fill = black!40!, draw = none] (-.485,-.5) rectangle (.485,.5);
    \filldraw[fill = white] (.5,.5) to[out = -135, in = 90] (.2,0) to[out = -90, in = 135]  (.5,-.5);
    \filldraw[fill = white] (-.5,.5) to[out =-45, in = 90] (-.2,0) to[out = -90, in = 45] (-.5,-.5);
\end{tikzpicture}   ~=~ \begin{tikzpicture}[baseline = -.1cm]
    \filldraw[fill = black!40!] (-.5,-.5) to[out = 90, in = 180] (0,-.2) to[out = 0, in = 90] (.5,-.5);
    \filldraw[fill = black!40!] (-.5,.5) to[out = -90, in = 180] (0,.2)  to[out = 0, in = -90] (.5,.5);
\end{tikzpicture} ~+~ \begin{tikzpicture}[baseline = -.1cm]
\draw[fill = black!40!, draw = none] (-.35,-.5) rectangle (.35,.5);
\draw (-.35,-.5) -- (-.35,.5);
\draw (.35,-.5) -- (.35,.5);
\draw[fill = white] (-.5,-.3) rectangle (.5,.3);
\node at (0,0) {\large{{$P$}}};
\end{tikzpicture} ~+~ \begin{tikzpicture}[baseline = -.1cm]
\draw[fill = black!40!, draw = none] (-.35,-.5) rectangle (.35,.5);
\draw (-.35,-.5) -- (-.35,.5);
\draw (.35,-.5) -- (.35,.5);
\draw[fill = white] (-.5,-.3) rectangle (.5,.3);
\node at (0,0) {\large{{$Q$}}};
\end{tikzpicture}
\end{equation}
Because $\mathcal{V}$ is a singly-generated YBPA, we know that there exists Relation 0 when $R_i = P$. Since, the diagram obtained by rotating $P$ 180 degrees is still a minimal idempotent orthogonal to $\begin{tikzpicture}[baseline = -.1cm, scale = .6]
    \filldraw[fill = black!40!] (-.5,-.5) to[out = 90, in = 180] (0,-.2) to[out = 0, in = 90] (.5,-.5);
    \filldraw[fill = black!40!] (-.5,.5) to[out = -90, in = 180] (0,.2)  to[out = 0, in = -90] (.5,.5);
\end{tikzpicture}$, we know that this rotation must be equal to $P$ or $Q$. In the former case this implies that $c_P(a,b) = c_P(b,a)$ for all $a$ and $b,$ which gives a symmetric spin model by Definition \ref{def:sym}. When the 180-degree rotation of $P$ is $Q$, the spin model is non-symmetric. We begin with the classification of symmetric spin models:

\subsection{Classification of symmetric spin models}\label{section:sym}

Assume that the 180-degree rotation of $P$ is equal to itself. For $P$ and $Q$ we have the following proposition:
\begin{proposition}\label{prop:graph}
Let $P$ and $Q$ be the minimal idempotents of a singly-generated YBPA orthogonal to $\begin{tikzpicture}[baseline = -.1cm, scale = .6]
    \filldraw[fill = black!40!] (-.5,-.5) to[out = 90, in = 180] (0,-.2) to[out = 0, in = 90] (.5,-.5);
    \filldraw[fill = black!40!] (-.5,.5) to[out = -90, in = 180] (0,.2)  to[out = 0, in = -90] (.5,.5);
\end{tikzpicture}$. Then the matrices $C_P$ and $C_Q$ for a symmetric spin model are the adjacency matrices of graphs $\Gamma$ and $\Gamma',$ respectively. Moreover, $\Gamma' = \Gamma^c$, the graph complement of $\Gamma.$
\end{proposition}

\begin{proof}
We first prove $C_P$ is the adjacency matrix of a graph. Since this is a symmetric spin model, $c_P(a,b) = c_P(b,a)$ for all $a$ and $b.$ Since $P$ is an idempotent, we have that $P^2 = P.$ In diagrammatic terms, this means that  
\[\includegraphics[valign = c, scale = .7]{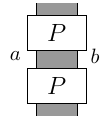} ~=~ \left(\!\xdiag[P]\!\right)^2 ~=~ \xdiag[P]\]
for any $a,b\in S,$ which implies $c_P(a,b) \in \{0,1\}$ for all $a$ and $b.$ Moreover, since $P$ is minimal, that means that
\[ \begin{tikzpicture}[baseline = -.1cm]
\draw[fill = black!40!, draw = none] (-.35,-.5) rectangle (.35,.5);
\draw (-.35,-.5) -- (-.35,.5);
\draw (.35,-.5) -- (.35,.5);
\draw[fill = white] (-.5,-.3) rectangle (.5,.3);
\node at (0,0) {\large{{$P$}}};
\end{tikzpicture} \cdot \begin{tikzpicture}[baseline = -.1cm]
    \filldraw[fill = black!40!] (-.5,-.5) to[out = 90, in = 180] (0,-.2) to[out = 0, in = 90] (.5,-.5);
    \filldraw[fill = black!40!] (-.5,.5) to[out = -90, in = 180] (0,.2)  to[out = 0, in = -90] (.5,.5);
\end{tikzpicture} ~=~ \includegraphics[valign = c, scale = .8]{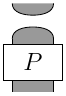} ~=~ 0 \]
which implies that $c_P(a,a) = 0$ for all $a\in S.$  Thus, $C_P$ is a symmetric matrix whose entries consist of only 0s and 1s with 0s on the diagonal. In other words, $C_P$ is the adjacency matrix of some graph, $\Gamma,$ with vertices indexed by the elements of $S$ and 
\[
c_P(a,b) = \left\{ \begin{array}{cl}
    1 & a \text{ adjacent to } b  \\
    0 & \text{else}.
\end{array} \right.
\]
By using an identical argument, we see that $C_Q$ is the adjacency matrix of some graph $\Gamma'.$

By (\ref{eq:qrela}), we know that $C_Q = J - I - C_P$, where $J$ is the matrix of all ones and $I$ is the identity matrix. Thus, $C_P + C_Q = J - I,$ which implies that $c_P(a,b) = 1$ if and only if $c_Q(a,b) = 0$ and vice versa. Thus, by definition $C_Q$ is the adjacency matrix $\Gamma^c$, the graph complement of $\Gamma.$ This completes our proof. 
\end{proof}

Thus, as in \cite{Jae95}, there is an intimate connection between spin models and graphs. For ease, we make the following definitions:
\begin{definition}\label{def:Gspin}
A graph $\Gamma$ \emph{gives a spin model} for a singly-generated YBPA, $\mathcal{V}$, if the adjacency matrix of $\Gamma$ is $C_P$ or $C_Q.$
\end{definition}
\begin{remark}
If we began with $Q$ as our designated minimal idempotent instead of $P$, we would obtain the same spin model. Since the graph of the adjacency matrix given by $C_Q$ is the complement of $C_P$, we will only classify $\Gamma$ up to complementation.  Moreover, because singly-generated YBPAs are non-degenerate by definition this implies that $\Gamma$ has at least one vertex. For convenience, we will assume that $\Gamma$ has at least one vertex for the remainder of this paper.
\end{remark} 
By Proposition \ref{prop:onlyp}, in order to check that one of the defining relations of a singly-generated YBPA planar algebra holds for all $R_i$, it suffices to only check the case when $R_i = P$. As such, we will be justified in not verifying the relations hold when at least one of the $R_i \ne P$. We will say that $\Gamma$ \emph{satisfies} one of the defining relations of a singly-generated YBPA when the corresponding relations of the state sums hold with respect to all possible assignments of vertices to the exterior regions of the relation.

Suppose $\Gamma$ gives a spin model for $\mathcal{V}.$ Fix a labeling of the vertices of $\Gamma$ by the elements of $S$. We can think of a state sum of a diagram of $\mathcal{V}$ with respect to $(a_1,\dots,a_k)$ as counting certain subgraphs involving $a$, $b$, and $c$. For example, with respect to vertices $a$, $b,$ and $c$ the state sum of 
\begin{center}
\begin{tabular}{ccccc}
 \includegraphics[valign = c, scale = .6]{diagrams/pdf/unshadedmiddlep.pdf}
 &&
  is &&
 $\displaystyle \mathlarger{\mathlarger{\sum}}_{x\in S} ~\includegraphics[valign = c]{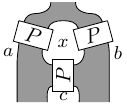} $
\end{tabular}
\end{center}
Since $x$ ranges over all vertices of $\Gamma$, we can think of this state sum as counting the number of $x$ making the following subgraph:
\begin{equation*}
    \includegraphics[valign = c, scale = .7]{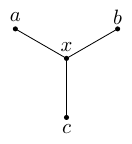}
\end{equation*}
In this case, this state sum counts the number of neighbors that $a$, $b$, and $c$ share.
\begin{proposition}\label{prop:strong}
Let $\Gamma$ be a graph and let $\mathcal{V}$ be a singly-generated YBPA. Then $\Gamma$ satisfies Relation 1b of $\mathcal{V}$ if and only if $\Gamma$ is regular and satisfies Relation 2b if and only if $\Gamma$ is strongly regular.
\end{proposition}

\begin{proof}
Suppose $\Gamma$ satisfies Relation 1b. By Proposition \ref{prop:graph}, we know that $C_P$ is the adjacency matrix of some graph, $\Gamma$. Then the state sum of Relation 1b tells us that 
\begin{equation}\label{eq:1b}
\includegraphics[valign = c]{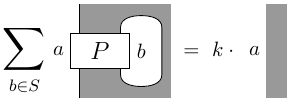}  
\end{equation}
for some $k\in\mathbb{C}$.  Thus, in graph theoretic terms, the left-hand side of the equation counts the number of neighbors of any vertex $a$.  Since $k$ is not dependent on $a$, the right-hand side of Relation 1b says that every vertex of $\Gamma$ has the same number of neighbors. In other words, $\Gamma$ is a regular (or 1-pt regular) graph by Definition \ref{def:1regular}. 

Suppose $\Gamma$ is regular such that each vertex has $k$ neighbors. Then for all $a$ and $b$ (\ref{eq:1b}) holds by definition of regularity, so by Proposition \ref{prop:atoms} Relation 1b of $\mathcal{V}$ must be equivalent to (\ref{eq:1b}). Thus, by definition, $\Gamma$ satisfies Relation 1b. 

Suppose $\Gamma$ satisfies Relation 2b. Then we know that
\begin{equation*}
  \begin{tikzpicture}[baseline = 0cm]
\draw[fill = black!40!, draw = none] (-.35,-.6) rectangle (1.85,.6);
\draw[fill = white] (.35,.3) to[out = 90, in = 180]  (.75,.55) 
							 to[out = 0, in = 90]    (1.15,.3)
					  		 to[out = -90, in = 90]  (1.15,-.3) 
					  		 to[out = -90, in = 0]   (.75, -.55)
					  		 to[out = 180, in = -90] (.35,-.3)
					  		 to[out = 90, in = -90]  (.35,.3);
\draw (-.35,-.6) -- (-.35,.6);
\draw[fill = white] (-.5,-.3) rectangle (.5,.3);
\draw (1.85,-.6) -- (1.85,.6);
\draw[fill = white] (1,-.3) rectangle (2,.3);
\node at (0,0) {\small{{$P$}}};
\node at (1.5,0) {\small{{$P$}}};
\end{tikzpicture} ~=~ c_1\begin{tikzpicture}[baseline = -.1cm]
    \draw[fill = black!40!, draw = none] (-.485,-.5) rectangle (.485,.5);
    \filldraw[fill = white] (.5,.5) to[out = -135, in = 90] (.2,0) to[out = -90, in = 135]  (.5,-.5);
    \filldraw[fill = white] (-.5,.5) to[out =-45, in = 90] (-.2,0) to[out = -90, in = 45] (-.5,-.5);
\end{tikzpicture} ~+~ c_2\begin{tikzpicture}[baseline = -.1cm]
    \filldraw[fill = black!40!] (-.5,-.5) to[out = 90, in = 180] (0,-.2) to[out = 0, in = 90] (.5,-.5);
    \filldraw[fill = black!40!] (-.5,.5) to[out = -90, in = 180] (0,.2)  to[out = 0, in = -90] (.5,.5);
\end{tikzpicture} ~+~ c_3\cdot\begin{tikzpicture}[baseline = -.1cm]
\draw[fill = black!40!, draw = none] (-.35,-.5) rectangle (.35,.5);
\draw (-.35,-.5) -- (-.35,.5);
\draw (.35,-.5) -- (.35,.5);
\draw[fill = white] (-.5,-.3) rectangle (.5,.3);
\node at (0,0) {\large{{$P$}}};
\end{tikzpicture} ~=~ k'\begin{tikzpicture}[baseline = -.1cm]
    \filldraw[fill = black!40!] (-.5,-.5) to[out = 90, in = 180] (0,-.2) to[out = 0, in = 90] (.5,-.5);
    \filldraw[fill = black!40!] (-.5,.5) to[out = -90, in = 180] (0,.2)  to[out = 0, in = -90] (.5,.5);
\end{tikzpicture} + \lambda\cdot\begin{tikzpicture}[baseline = -.1cm]
\draw[fill = black!40!, draw = none] (-.35,-.5) rectangle (.35,.5);
\draw (-.35,-.5) -- (-.35,.5);
\draw (.35,-.5) -- (.35,.5);
\draw[fill = white] (-.5,-.3) rectangle (.5,.3);
\node at (0,0) {\large{{$P$}}};
\end{tikzpicture} +\mu\cdot\begin{tikzpicture}[baseline = -.1cm]
\draw[fill = black!40!, draw = none] (-.35,-.5) rectangle (.35,.5);
\draw (-.35,-.5) -- (-.35,.5);
\draw (.35,-.5) -- (.35,.5);
\draw[fill = white] (-.5,-.3) rectangle (.5,.3);
\node at (0,0) {\large{{$Q$}}};
\end{tikzpicture} 
\end{equation*}
For any $a,b\in S,$ this implies that
\begin{equation}\label{2batoms}
\mathlarger{\mathlarger{\sum}}_{x\in S} \begin{tikzpicture}[baseline = 0cm]
\draw[fill = black!40!, draw = none] (-.35,-.6) rectangle (1.85,.6);
\draw[fill = white] (.35,.3) to[out = 90, in = 180]  (.75,.55) 
							 to[out = 0, in = 90]    (1.15,.3)
					  		 to[out = -90, in = 90]  (1.15,-.3) 
					  		 to[out = -90, in = 0]   (.75, -.55)
					  		 to[out = 180, in = -90] (.35,-.3)
					  		 to[out = 90, in = -90]  (.35,.3);
\draw (-.35,-.6) -- (-.35,.6);
\draw[fill = white] (-.5,-.3) rectangle (.5,.3);
\draw (1.85,-.6) -- (1.85,.6);
\draw[fill = white] (1,-.3) rectangle (2,.3);
\node at (0,0) {\small{{$P$}}};
\node at (1.5,0) {\small{{$P$}}};
\node at (-.7,0) {\small{{$a$}}};
\node at (.7,0) {\small{{$x$}}};
\node at (2.2,0) {\small{{$b$}}};
\end{tikzpicture} ~=~  k'\cdot\begin{tikzpicture}[baseline = -.1cm]
    \filldraw[fill = black!40!] (-.5,-.5) to[out = 90, in = 180] (0,-.2) to[out = 0, in = 90] (.5,-.5);
    \filldraw[fill = black!40!] (-.5,.5) to[out = -90, in = 180] (0,.2)  to[out = 0, in = -90] (.5,.5);
    \node at (-.5,0) {$a$};
    \node at (.5,0) {$b$};
\end{tikzpicture} ~+~ \lambda\cdot\xdiag[P] ~+~ \mu\cdot\xdiag[Q]   
\end{equation}
In graph-theoretic terms, we can think of the left-hand side of (\ref{2batoms}) as counting the number of subgraphs of one of the following forms:
\begin{center}
 \includegraphics[valign = c, scale = .7]{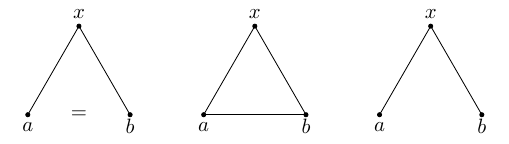}     
\end{center}
depending on how $a$ and $b$ relate to each other. For each one of the above subgraphs, note that exactly one of the diagrams on the right-hand side of (\ref{2batoms}) survives by Proposition \ref{prop:graph}. Thus, this says that when $a$ and $b$ are equal, the number of $x$ adjacent to $a$ is $k'$, which shows that $k' = k$. When, $a$ and $b$ are adjacent, only $P$ survives on the right-hand side. Thus, (\ref{2batoms}) tells us that the number of $x$ adjacent to $a$ and $b$ is some constant $\lambda$. Similarly, when $a$ and $b$ are not adjacent, the equation tells us that the number of $x$ adjacent to $a$ and $b$ is some constant $\mu$. Thus by Definition \ref{def:strong}, $\Gamma$ is strongly regular. 

Suppose $\Gamma$ is strongly regular with parameters $(n,k,\lambda,\mu)$. Then for all $a$ and $b$ (\ref{2batoms}) holds, and so by Proposition \ref{prop:atoms} $\Gamma$ satisfies Relation 2b. This completes our proof.  
\end{proof}

Thus, Relations 1a and 2a tell us that $C_P$ is the adjacency matrix for some graph, $\Gamma,$ while Relations 1b and 2b indicate that $\Gamma$ is strongly regular. We now describe the conditions Relations 3a and 3b put on $\Gamma$. 

\begin{lemma}\label{thm:3a}
Let $\Gamma$ be a strongly regular graph. Then $\Gamma$ satisfies Relation 3a if and only if $\Gamma$ is 3-point regular. 
\end{lemma}

\begin{proof}
Consider Relation 3a:
\begin{equation}\label{eq:unshadedmiddle2}
\begin{split}
     \includegraphics[valign = c,scale = .6]{diagrams/pdf/unshadedmiddlep.pdf} 
   & =~ \includegraphics[valign = c, scale = .95]{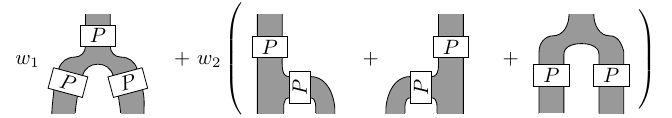} 
   \\
   & \includegraphics[valign = c,, scale = .95]{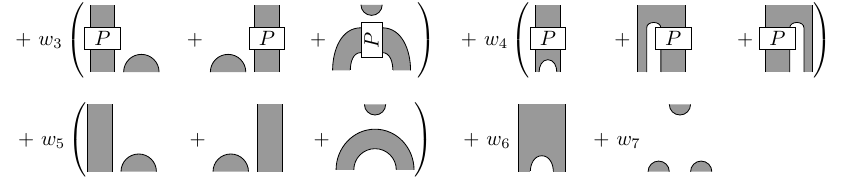} 
\end{split}
\end{equation}
Note that the right-hand side of Relation 3a must be of the above form since the left-hand side is invariant under 2-click rotations. Let $\Gamma$ be a graph and $a,$ $b,$ and $c$ be vertices of $\Gamma$. Then (\ref{eq:unshadedmiddle2}) gives us the following state sum equation: 
\begin{equation}\label{eq:unshadedmiddle}
\begin{split}
    \mathlarger{\mathlarger{\sum}}_{x\in S} ~\includegraphics[valign = c, scale = .9]{diagrams/pdf/unshadedmiddlepatoms.pdf}  
   & =~ \includegraphics[valign = c, scale = .95]{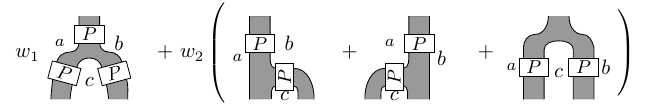} 
   \\
   & \includegraphics[valign = c, scale = .95]{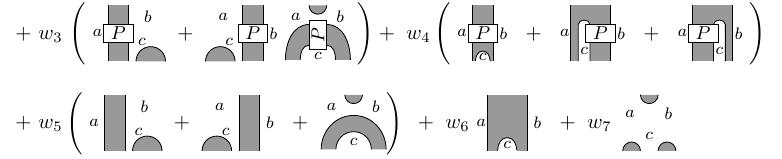} 
\end{split}
\end{equation}
For some fixed $a$, $b$ and $c,$ we know the left-hand side of (\ref{eq:unshadedmiddle}) can be interpreted as the number of $x$ adjacent to $a$, $b$, and $c$. Note that for any fixed $a$, $b$, and $c$ the state sums on the right-hand side can be either $1$ or $0$. Moreover, that value is only dependent on how $a,$ $b,$ and $c$ relate to each other. Thus, the number of vertices connected to $a,$ $b,$ and $c$ is only dependent on how those vertices relate to each other. This is exactly the statement of Definition \ref{def:nptreg} when $n = 3$. Thus if $\Gamma$ satisfies Relation 3a then $\Gamma$ is 3-point regular. 

If $\Gamma$ is 3-point regular, then the number of $x$ adjacent to any $a,$ $b,$ and $c$ depends only on how those points relate to each other. Consider 
\begin{equation}\label{eq:3a2}
    \includegraphics[valign = c]{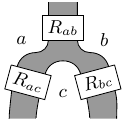} 
\end{equation}
Note that there (\ref{eq:3a2}) is either 1 or 0 for all $a,b,c\in S$ and $\xdiagscaled[R_{ab}] \in \left\{\begin{tikzpicture}[baseline = -.1cm, scale = .6]
\draw[fill = black!40!, draw = none] (-.35,-.5) rectangle (.35,.5);
\draw (-.35,-.5) -- (-.35,.5);
\draw (.35,-.5) -- (.35,.5);
\draw[fill = white] (-.5,-.3) rectangle (.5,.3);
\node at (0,0) {\tiny{{$P$}}};
\end{tikzpicture}, ~\begin{tikzpicture}[baseline = -.1cm, scale = .6]
\draw[fill = black!40!, draw = none] (-.35,-.5) rectangle (.35,.5);
\draw (-.35,-.5) -- (-.35,.5);
\draw (.35,-.5) -- (.35,.5);
\draw[fill = white] (-.5,-.3) rectangle (.5,.3);
\node at (0,0) {\tiny{{$Q$}}};
\end{tikzpicture}, ~\begin{tikzpicture}[baseline = -.1cm, scale = .6]
    \filldraw[fill = black!40!] (-.5,-.5) to[out = 90, in = 180] (0,-.2) to[out = 0, in = 90] (.5,-.5);
    \filldraw[fill = black!40!] (-.5,.5) to[out = -90, in = 180] (0,.2)  to[out = 0, in = -90] (.5,.5);
\end{tikzpicture} \right\}$ by setting
    \[ \begin{tikzpicture}[baseline = -.05cm]
        \draw[fill = black!40!, draw = none] (-.35,-.5) rectangle (.35,.5);
        \draw (-.35,-.5) -- (-.35,.5);
        \draw (.35,-.5) -- (.35,.5);
        \draw[fill = white] (-.5,-.3) rectangle (.5,.3);
        \node at (0,0) {\large{$R_{ab}$}};
        \end{tikzpicture} 
        = \left\{
    \begin{array}{cl}
        \begin{tikzpicture}[baseline = -.1cm, scale = .6]
    \filldraw[fill = black!40!] (-.5,-.5) to[out = 90, in = 180] (0,-.2) to[out = 0, in = 90] (.5,-.5);
    \filldraw[fill = black!40!] (-.5,.5) to[out = -90, in = 180] (0,.2)  to[out = 0, in = -90] (.5,.5);
\end{tikzpicture} & \text{if } a~=~b \\
        \\
        \begin{tikzpicture}[baseline = -.1cm, scale = .6]
\draw[fill = black!40!, draw = none] (-.35,-.5) rectangle (.35,.5);
\draw (-.35,-.5) -- (-.35,.5);
\draw (.35,-.5) -- (.35,.5);
\draw[fill = white] (-.5,-.3) rectangle (.5,.3);
\node at (0,0) {\tiny{{$P$}}};
\end{tikzpicture}  & \text{if } a \text{ is adjacent to } b \\
        \\
        \begin{tikzpicture}[baseline = -.1cm, scale = .6]
\draw[fill = black!40!, draw = none] (-.35,-.5) rectangle (.35,.5);
\draw (-.35,-.5) -- (-.35,.5);
\draw (.35,-.5) -- (.35,.5);
\draw[fill = white] (-.5,-.3) rectangle (.5,.3);
\node at (0,0) {\tiny{{$Q$}}};
\end{tikzpicture} & \text{if } a \text{ is not adjacent to } b     
    \end{array} \right.
    \]
and similarly for $R_{ac}$ and $R_{bc}$. Moreover, there is exactly one diagram of the form (\ref{eq:3a2}) that is is equal to 1 for any vertices $a$, $b,$ and $c.$ Since $\Gamma$ is 3-point regular, then  
\begin{equation}\label{eq:3a3}
    \mathlarger{\mathlarger{\sum}}_{x\in S} ~\includegraphics[valign = c, scale = .9]{diagrams/pdf/unshadedmiddlepatoms.pdf} 
  = \mathlarger{\mathlarger{\sum}}_{i} ~q_i\cdot \includegraphics[valign = c, scale = .9]{diagrams/pdf/shadedmiddlepatomsb.pdf} 
\end{equation}
where $i$ ranges over all possible relationships among $a,$ $b,$ and $c,$ in $\Gamma$ and $q_i$ is the corresponding number of $x$ adjacent to $a,$ $b,$ and $c.$ Using (\ref{eq:qrela}), we can transform (\ref{eq:3a3}) into an equation of the form (\ref{eq:unshadedmiddle}). Thus, by  Proposition \ref{prop:atoms}, (\ref{eq:unshadedmiddle2}) must hold in $\mathcal{V}$. Thus, $\Gamma$ satisfies Relation 3a. Hence, $\Gamma$ satisfies Relation 3a if and only if $\Gamma$ is 3-point regular. 
\end{proof}

Thus, in order for $\Gamma$ to give a spin model, it is necessary to be 3-point regular. We will see, however, that this is not sufficient to satisfy Relation 3b. We will now determine the necessary and sufficient conditions for $\Gamma$ to satisfy Relation 3b. 

\begin{lemma}\label{lem:tfree}
Let $\Gamma$ give a spin model for a singly-generated YBPA, $\mathcal{V}$. Then $\Gamma$ is triangle-free if and only if the following relation holds in $\mathcal{V}$: 
\begin{equation}\label{eq:tfree}
     \includegraphics[valign = c, scale = .6]{diagrams/pdf/shadedmiddlep.pdf}  ~=~ 0.
\end{equation}
Moreover, if $\Gamma$ is a triangle-free graph then $\Gamma$ satisfies Relation 3b. 
\end{lemma}

\begin{proof}
Suppose $\Gamma$ is triangle-free. Then 
\begin{equation}\label{eq:tfree2}
     \includegraphics[valign = c, scale = .9]{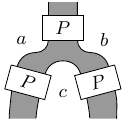}  ~=~ 0
\end{equation}
for all $a$, $b$, and $c$. Thus, Proposition \ref{prop:atoms} tells us that (\ref{eq:tfree}) holds in $\mathcal{V}$. If (\ref{eq:tfree}) holds in $\mathcal{V}$, then (\ref{eq:tfree2}) holds for all $a,b,c\in S$, which implies $\Gamma$ is triangle-free by definition. Thus, $\Gamma$ is triangle-free if and only if (\ref{eq:tfree}) holds in $\mathcal{V}$. Moreover, we can conclude that $\Gamma$ satisfies Relation 3b in this case, as (\ref{eq:tfree}) is equivalent to Relation 3b by Proposition \ref{prop:atoms}. 

\end{proof}

\begin{lemma}\label{lem:lfree}
Let $\Gamma$ give a spin model for a singly-generated YBPA, $\mathcal{V}$. Then $\Gamma$ is a strongly regular, $\Lambda$-free graph with regularity parameter $k$ if and only if the following relation holds in $\mathcal{V}$: 
\begin{equation}\label{eq:lfree}
     \includegraphics[valign = c, scale = .87]{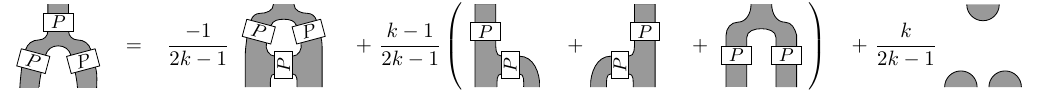}
\end{equation}
Moreover, if $\Gamma$ is a strongly regular and $\Lambda$-free graph then $\Gamma$ satisfies Relation 3b.
\end{lemma}

\begin{proof} 
Suppose $\Gamma$ is a strongly regular and $\Lambda$-free graph. By Proposition \ref{prop:lfree}, $\Gamma$ is a disjoint union of $m$ complete graphs of size $k+1$. If $k = 0$, $\Gamma$ is a disjoint union of vertices and $(\ref{eq:lfree})$ becomes
\begin{equation}\label{eq:lfreek0}
     \includegraphics[valign = c, scale = .7]{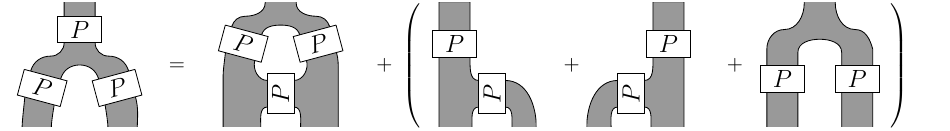}
\end{equation}
Since there are no edges in this graph, though, note that the state sum for each diagram appearing in (\ref{eq:lfreek0}) is 0 for all $a$, $b$, and $c$. Thus, by Proposition \ref{prop:atoms}, we see that each term must be 0 and by extension that (\ref{eq:lfree}) holds in $\mathcal{V}$. 

If $k = 1$, then $\Gamma$ is a disjoint union of complete graphs of size 2 and (\ref{eq:lfree}) becomes
\begin{equation}\label{eq:lfreek1}
     \includegraphics[valign = c, scale = .7]{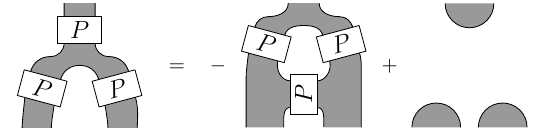}
\end{equation}
Consider the state sum for the diagrams in (\ref{eq:lfreek1}) with respect to $(a,b,c).$ If $a,$ $b,$ and $c$ are not all equal, then the state sum of all three diagrams is 0 since $k = 1.$ If $a=b=c$, the state sum for the left-hand diagram is 0 while the state sum for both right-hand diagrams is 1. Since $0 = -1 + 1$, we can conclude by Proposition \ref{prop:atoms} that this relation must hold in $\mathcal{V}$.  

Suppose now that $k \geq 2.$ By Corollary \ref{lfree}, $\Gamma$ is 3-point regular with strongly regular parameters $(m(k+1),k,k-1,0)$ and 3-point regular parameters $q_3 = k-2$ and $q_2=q_1=q_0 = 0$. Thus, by Lemma \ref{thm:3a}, we know that $\Gamma$ satisfies Relation 3a. By explicitly solving the state sums for each possible relationship among $a,$ $b,$ and $c$, we find that
\begin{equation}\label{eq:lfree2}
     \includegraphics[valign = c,scale = .87]{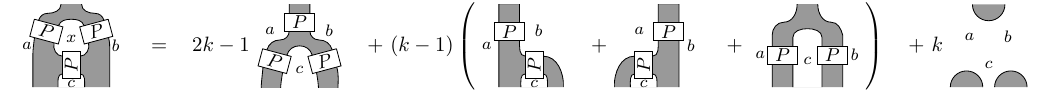}
\end{equation}
holds for all $a,$ $b,$ and $c.$ By rearranging (\ref{eq:lfree2}), Proposition \ref{prop:atoms} tells us that (\ref{eq:lfree}) holds in $\mathcal{V}$ when $k \geq 2$. 

Suppose (\ref{eq:lfree}) holds in $\mathcal{V}$ and let $\Gamma$ give a spin model for $\mathcal{V}$. Then we can rearrange (\ref{eq:lfree}) in the form of Relation 3a, which implies that $\Gamma$ is 3-point regular. Thus, $\Gamma$ is also strongly regular. Let $a,$ $b,$ and $c$ be vertices of $\Gamma$. The state sum of (\ref{eq:lfree}) where $a = b = c$ tell us that each vertex has $k$ neighbors, so the regularity parameter of $\Gamma$ is $k$. Note that if $k = 0$ or $1$, then $\Gamma$ must be $\Lambda$-free. Suppose now that $k \geq 2.$ Suppose $\Gamma$ has $n$ vertices. If $k = n - 1$ then $\Gamma$ is complete and must be $\Lambda$-free. Suppose now that $\Gamma$ is not complete. Then there exist distinct points $a$ and $b$ that are not adjacent by definition. Consider the state sum of (\ref{eq:lfree}) where $a = c$ and $a$ and $b$ are distinct, non-adjacent points. Then the left-hand side of the equation counts the number of common neighbors shared by $a$ and $b$. By inspection, the right-hand side is 0, so $a$ and $b$ share no common neighbors. Thus, $\Gamma$ is $\Lambda$-free by definition, and in all cases $\Gamma$ is strongly regular and $\Lambda$-free. If $\Gamma$ is strongly regular and triangle-free, we know that Relation 3b must be equivalent to (\ref{eq:lfree}) by Proposition \ref{prop:atoms}. Thus, $\Gamma$ satisfies Relation 3b, which completes our proof. 
\end{proof}

\begin{lemma}\label{lem:nonzero}
Suppose $\Gamma$ satisfies Relation 3b for some $\mathcal{V}$ and that $\Gamma$ is a strongly regular graph that is neither triangle-free, $\Lambda$-free, anti-$\Lambda$-free, nor anti-triangle-free. Then a non-zero constant multiple of 
\begin{center}
    \includegraphics[valign = c, scale = .7]{diagrams/pdf/unshadedmiddlep.pdf}
\end{center}
must appear in Relation 3b. 
\end{lemma}

\begin{proof}
Suppose $\Gamma$ satisfies Relation 3b. Then Relation 3b gives us the following relation on the state sums for all $a,$ $b,$ and $c$:
\begin{equation}\label{eq:shadedmiddle}
\begin{split}
     \includegraphics[valign = c, scale = .75]{diagrams/pdf/shadedmiddlepatoms.pdf}  
   & =~ \includegraphics[valign = c,scale = .9]{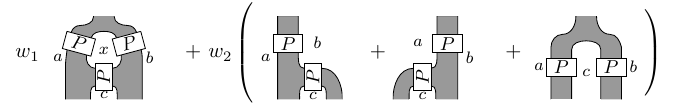} 
   \\
   & \includegraphics[valign = c,scale = .9]{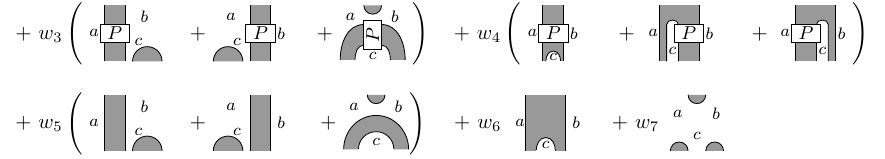} 
\end{split}
\end{equation}
Suppose that $w_1 =0$. If $\Gamma$ is a strongly regular graph that is neither triangle-free, $\Lambda$-free, anti-$\Lambda$-free, nor anti-triangle-free, then up to rotation the following graphs appear as subgraphs induced by some triple of vertices of $\Gamma$:
\begin{center}
    \includegraphics[valign = c, scale = .8]{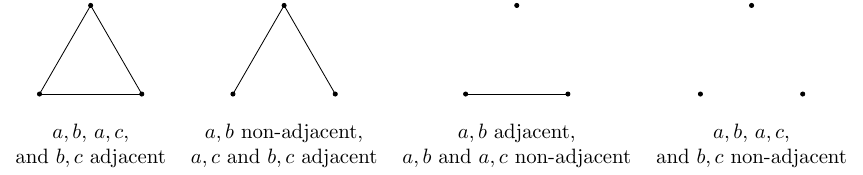}
    
    \vspace{.1in}
    
    \includegraphics[valign = c,scale = .8]{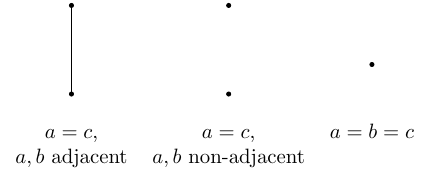}
\end{center}
Since $w_1 = 0$ and $\Gamma$ is strongly regular, we can calculate the state sum of each of the diagrams in (\ref{eq:shadedmiddle}) with respect to $a,$ $b,$ and $c$ in each scenario. The resulting system of linear equations in terms of the $w_i$ is
\begin{equation*}
\begin{array}{ccccccccc}
1 & = & 3w_2 && + 3w_4 &&& + w_7 \\
0 & = &  w_2 && + 2w_4 &&& + w_7  \\
0 & = &&&  w_4 &&& + w_7 \\
0 & = &&&&&&  w_7 \\
0 & = &w_2 &+ w_3 &+ 2w_4 &&+w_6 &+ w_7\\
0 & = &&&&&+ w_6 &+ w_7 \\
0 & = &&&& w_5 &+ w_6 &+ w_7 
\end{array}
\end{equation*}
The above system of equations, however, is inconsistent, so $\Gamma$ does not satisfy Relation 3b, a contradiction. Thus $w_1 \ne 0$ and so a non-zero multiple of 
\begin{center}
    \includegraphics[valign = c, scale = .7]{diagrams/pdf/unshadedmiddlep.pdf}
\end{center}
must appear in Relation 3b, as desired. 
\end{proof}

\begin{lemma}\label{lem:3b}
Let $\Gamma$ give a spin model for $\mathcal{V}$, and suppose $\Gamma$ is a strongly regular graph that is neither triangle-free, $\Lambda$-free, anti-$\Lambda$-free, nor anti-triangle-free. Then $\Gamma$ satisfies Relation 3b if and only if $\Gamma$ is 3-point regular with 3-point regular parameters satisfying $q_3 - 3q_2 +3q_1 - 3q_0 \ne 0$.
\end{lemma}

\begin{proof}
Let $\mathcal{V}$ and $\Gamma$ be as described above. Then $\Gamma$ has all subgraphs listed in Lemma \ref{lem:nonzero}. Suppose that $\mathcal{V}$ satisfies Relation 3b. By Lemma \ref{lem:nonzero} we know that a non-zero multiple of
\begin{center}
    \includegraphics[valign = c, scale = .7]{diagrams/pdf/unshadedmiddlep.pdf}
\end{center}
must appear in Relation 3b. Thus, we can rearrange Relation 3b into the form of Relation 3a. Hence, $\mathcal{V}$ also satisfies Relation 3a, and so by Lemma \ref{thm:3a} $\Gamma$ must be 3-point regular. Since $\Gamma$ is 3-point regular and neither triangle-free, $\Lambda$-free, anti-$\Lambda$-free, nor anti-triangle-free, we can explicitly solve for the values of $w_i$ which gives the relation 
\begin{equation}\label{eq:unshadedmiddleb}
\begin{array}{rl}
     \includegraphics[valign = c, scale = .4]{diagrams/pdf/unshadedmiddlep.pdf}  
   & = \includegraphics[valign = c, scale = .87]{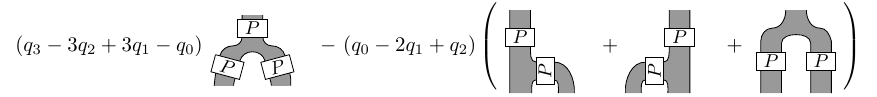}
   \\
   \\
   & \includegraphics[valign = c,, scale = .87]{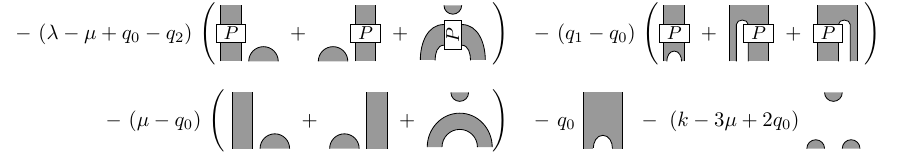} 
\end{array}
\end{equation}
Since this relation was obtained from Relation 3b, the coefficient of 
\begin{center}
    \includegraphics[valign = c, scale = .7]{diagrams/pdf/shadedmiddlep.pdf}
\end{center}
cannot be $0$, so $q_3 - 3q_2 + 3q_1 - q_0 \ne 0$. Thus, if $\mathcal{V}$ satisfies Relation 3b, $\Gamma$ must be 3-point regular with $q_3 - 3q_2 + 3q_1 - q_0 \ne 0.$ 

Now suppose that $\Gamma$ is 3-point regular with $q_3 - 3q_2 +3q_1 - q_0 \ne 0.$ Then by Lemma \ref{thm:3a}, $\mathcal{V}$ satisfies Relation 3a. Thus, we again obtain (\ref{eq:unshadedmiddleb}). By rearranging that equation, we have
\begin{equation}\label{eq:shadedmiddleb}
\begin{array}{rl}
     \includegraphics[valign = c, scale = .4]{diagrams/pdf/shadedmiddlep.pdf}  
   & = ~ \includegraphics[valign = c, scale = .87]{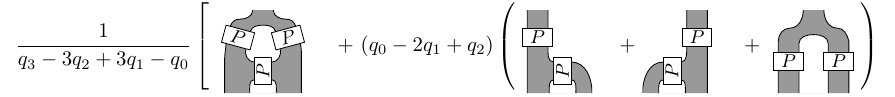} 
   \\
   \\
   & \includegraphics[valign = c, scale = .87]{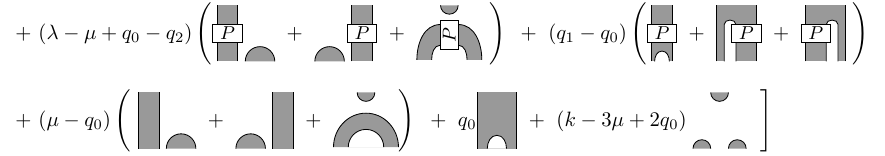} 
\end{array}
\end{equation}
which is well-defined since $q_3 - 3q_2 + 3q_1 - q_0 \ne 0$ by assumption. Thus, by Proposition \ref{prop:atoms} $\Gamma$ satisfies Relation 3b, which completes our proof.
\end{proof}

\begin{proposition}\label{prop:3b2}
Let $\Gamma$ be a strongly regular graph. Then $\Gamma$ satisfies Relation 3b if and only if up to complementation $\Gamma$ is  
\begin{enumerate}[label = \roman{enumi}.]
    \item $\Lambda$-free,
    \item triangle-free, or
    \item a 3-point regular graph that is neither triangle-free, $\Lambda$-free, anti-$\Lambda$-free, nor anti-triangle-free such that $q_3 - 3q_2 +3q_1 - q_0 \ne 0.$ 
\end{enumerate}
\end{proposition}

\begin{proof}
By Lemma \ref{lem:tfree}, Lemma \ref{lem:lfree}, and Lemma \ref{lem:3b}, we know that $\Gamma$ satisfies Relation 3b if $\Gamma$ is in the above list. Suppose $\Gamma$ satisfies Relation 3b. If Relation 3b is of the form of Lemma \ref{lem:tfree} or Lemma \ref{lem:lfree}, then $\Gamma$ must be triangle-free or $\Lambda$-free. Suppose now that Relation 3b is not of that form. Thus, up to complementation $\Gamma$ is neither triangle-free, $\Lambda$-free, anti-$\Lambda$-free, nor anti-triangle-free. Thus, the coefficient of
\begin{center}
    \includegraphics[valign = c, scale = .6]{diagrams/pdf/unshadedmiddlep.pdf}
\end{center}
is not 0 by Lemma \ref{lem:3b}. Hence, we can rearrange Relation 3b in the form of Relation 3a, which implies that $\Gamma$ is 3-point regular. Moreover, we know that the Relation 3a must be equivalent to (\ref{eq:unshadedmiddleb}). Since this relation was obtained from Relation 3b, we know that $q_3 - 3q_2 + 3q_1 - q_0 \ne 0$. By Lemma \ref{free}, up to complementation every 3-point regular graph is either triangle-free, $\Lambda$-free, or neither triangle-free, anti-$\Lambda$-free, anti-$\Lambda$-free, nor anti-triangle-free. Thus, if $\Gamma$ satisfies Relation 3b it is $\Lambda$-free, triangle-free, or a 3-point regular graph that is neither triangle-free, $\Lambda$-free, anti-$\Lambda$-free, nor anti-triangle-free such that $q_3 - 3q_2 +3q_1 - q_0 \ne 0.$ This completes our proof.
\end{proof}

\begin{theorem}\label{main}
Let $\mathcal{V}$ be a singly-generated Yang-Baxter planar algebra. Then $\Gamma$ gives a symmetric spin model for $\mathcal{V}$ if and only if up to complementation $\Gamma$ is one of the following:
\begin{enumerate}[label = \roman{enumi}.]
\item the pentagon,
\item a disjoint union of complete graphs, or
\item a 3-point regular graph with $q_3 - 3q_2 + 3q_1 - q_0 \ne 0$. 
\end{enumerate}
\end{theorem}

\begin{proof} 
$\Gamma$ gives a symmetric spin model for $\mathcal{V}$ if and only if it satisfies Relations 1a, 1b, 2a, 2b, 3a, and 3b and has the property that the 2-click rotation of $P$ is equal to $P$. If $\Gamma$ is in the above list, Proposition \ref{prop:graph}, Proposition \ref{prop:strong}, Lemma \ref{thm:3a}, and Proposition \ref{prop:3b2} together tell us that $\Gamma$ gives a symmetric spin model.

Suppose $\Gamma$ gives a symmetric spin model for $\mathcal{V}$. By assumption, the 2-click rotation of $P$ is equal to $P$. Under this assumption, Proposition \ref{prop:graph} tell us $\Gamma$ satisfies Relations 1a and 1b if and only if $\Gamma$ is a graph of at least one vertex. By Proposition \ref{prop:strong}, we know that $\Gamma$ satisfies 1b and 2b if and only if it is a strongly regular graph. Lemma \ref{thm:3a} tells us that $\Gamma$ satisfies Relation 3a if and only if it is three-point regular. 

Thus, suppose that $\Gamma$ is 3-point regular, which implies that $\Gamma$ is also strongly regular. By Proposition \ref{prop:3b2}, we know that $\Gamma$ satisfies Relation 3b if and only if up to complementation $\Gamma$ is
\begin{enumerate}[label = \roman{enumi}.]
    \item $\Lambda$-free,
    \item triangle-free, or
    \item neither triangle-free, $\Lambda$-free, anti-$\Lambda$-free, nor anti-triangle-free such that $q_3 - 3q_2 +3q_1 - q_0 \ne 0.$ 
\end{enumerate}
If $\Gamma$ is $\Lambda$-free, then it is a disjoint union of complete graphs by Proposition \ref{prop:lfree}. Suppose now that $\Gamma$ is triangle-free but not $\Lambda$-free. Thus, $k \geq 2$ necessarily. If $k = 2$, then $\Gamma$ is the pentagon by Lemma \ref{lem:cycle} (Note that the complement of the square is a disjoint union of 2 complete graphs of size 2). Else, if $k\geq 3$, then $\Gamma$ has that $q_3 = q_2 = q_1 = 0$ necessarily. Since $k \geq 3$, every vertex $x$ is adjacent to some triple of distinct vertices $a,$ $b,$ and $c,$ which are pairwise non-adjacent as $\Gamma$ is triangle-free. Thus, $a$, $b,$ and $c$ share at least one common neighbor, therefore $q_0 > 0$ by the 3-point regularity of $\Gamma$. Hence, $q_3 - 3q_2 + 3q_1 - q_0 \ne 0$. Assume that neither $\Gamma$ nor its complement are triangle-free nor $\Lambda$-free. Then $\Gamma$ is neither triangle-free, $\Lambda$-free, anti-$\Lambda$-free, nor anti-triangle-free and Lemma \ref{lem:3b} tells us that $\Gamma$ satisfies 3b if and only if $q_3 - 3q_2 +3q_1 - q_0 \ne 0.$ Since we have exhausted all possible cases up to complementation, $\Gamma$ is either the pentagon, a disjoint union of complete graphs, or a 3-point regular graph with $q_3 - 3q_2 + 3q_1 - q_0 \ne 0,$ which completes our proof.  
\end{proof}

Using the classification of singly-generated Yang-Baxter planar algebras found in \cite{Liu15}, we can make the following corollary to this theorem:

\begin{corollary}
Let $\Gamma$ give a symmetric spin model for a singly-generated Yang-Baxter planar algebra, $\mathcal{V}$. Then $\mathcal{V}$ is a TLJ planar algebra when $\Gamma$ or $\Gamma^c$ is complete, a specialization of Bisch-Jones when $\Gamma$ or $\Gamma^c$ is a disjoint union of at least two $K_n$, $n > 1$, and is a Kauffman polynomial planar algebra otherwise. 
\end{corollary}

\begin{proof}
By Theorem \ref{Liu}, we know that any singly-generated Yang-Baxter planar algebra is isomorphic to either a TLJ planar algebra, a Bisch-Jones planar algebra, a Kauffman polynomial planar algebra, or a Liu planar algebra. Because these planar algebras give a symmetric spin model, however, this implies they cannot be Liu planar algebras. Let $mK_n$ be the disjoint union of $m$ complete graphs of size $n$. Then by Proposition \ref{inj}, we know that the dimension of $V_3$ for each value of $m$ and $n$ is
\begin{center}
\begin{tabular}{cc|c}
    \multicolumn{2}{c|}{$mK_n$} & $\dim V_3$  \\
    \hline
    $m = 1$ & $n\in\mathbb{N}$ & 5 \\
    $m = 2$ & $n = 2$ & 10 \\
    $m = 2$ & $n>2$ & 11 \\
    $m>2$ & $n>2$ & 12 \\
\end{tabular}
\end{center}
Thus, when $\Gamma$ is complete it gives a spin model for a TLJ planar algebra, and a Bisch-Jones planar algebra when $\Gamma = mK_n$ for $m>1$ by the classification of \cite{BJ00}. If $\Gamma$ is the pentagon, then $\dim V_3 =13$, so it is the unique planar algebra described in \cite{BJ03}, which is a Kauffman polynomial planar algebra by \cite{Liu15}. Otherwise, the planar algebra has $\dim V_3 = 14$ or $15,$ which must be a Kauffman polynomial planar algebra by the classification of \cite{Liu15} and the fact that the 2-click rotation of $P$ is itself. This completes our proof. 
\end{proof}

\subsubsection{Consequences of Theorem \ref{main}}

The classification of symmetric spin models in Theorem \ref{main} gives rise to a number of interesting facts. First, the results of that theorem can be seen as a partial classification of 3-point regular graphs. That is, a 3-point regular graph could give a spin model (i.e. be either the pentagon, a disjoint union of complete graphs, or have the property that $q_3 - 3q_2 + 3q_1 - q_0 \ne 0$)  or it could not (i.e. have the property that $q_3 - 3q_2 + 3q_1 - q_0 = 0$). In \cite{CGS78}, it is shown that up to complementation every 3-point regular graph is either the pentagon, a disjoint union of complete graphs, a Smith graph, or a negative Latin square graph. Rather surprisingly, the results of this theorem give an identical partial classification of 3-point regular graphs. In particular, every Smith graph has the property that $q_3 - 3q_2 + 3q_1 - q_0 = 0$ while no (non-complete) negative Latin Square graph has that property!

There are also a number of interesting graphs which satisfy exactly one of Relation 3a or Relation 3b (in addition to Relations 1a, 1b, 2a, and 2b). Let $\Gamma$ be a graph that satisfies Relation 3a but not Relation 3b. Then we know by Lemma \ref{thm:3a} that $\Gamma$ is 3-point regular. Since $\Gamma$ does not satisfy Relation 3b, we know that $\Gamma$ must be a 3-point regular graph that is neither triangle-free, $\Lambda$-free, anti-$\Lambda$-free, nor anti-triangle-free such that $q_3 - 3q_2 +3q_1 - q_0 = 0.$  These are exactly the Smith graphs mentioned in \cite{CGS78}. Examples of such a graph include the Schl\"afli  graph and McLaughlin graph. Let $\Gamma$ be a graph that satisfies Relation 3b but not Relation 3a (in addition to Relations 1a, 1b, 2a, and 2b). Since it does not satisfy Relation 3a, it cannot be 3-point regular. Thus, by Proposition \ref{prop:3b2}, $\Gamma$ must be strongly regular and triangle-free but not 3-point regular. Examples of such a graph include the Petersen graph. Following a suggestion by Jones, Yunxiang Ren investigated the skein theory of the planar algebra coming from the Petersen graph in order to understand what skein theory looks like beyond Yang-Baxter relations.  Our work suggests that it would also be interesting to study the skein theory for the Schl\"afli and McLaughlin graphs which also go beyond the Yang-Baxter relation albeit in a different way.

These spin models can also be used to define interesting fiber functors. For instance, the spin models of the Clebsch graph defines a fiber functor from $\Rep(S_4)$ which is not symmetric. More about this fiber functors and the others that arise will follow in a forthcoming paper. 

\subsection{Classification of non-symmetric spin models}

In the previous section, we assumed that the 2-click rotation of $P$ was equal to itself, which gave a symmetric spin model. Assume now that the spin model is not symmetric, which implies that the the 2-click rotation of $P$ is equal to $Q$ and $P \ne Q.$ Like in the symmetric case, we will be dealing with graphs. Because the spin model is not symmetric, however, the graph will be directed. More specifically, they will be tournaments.

\begin{proposition}\label{prop:digraph}
Let $\mathcal{V}$ be a singly-generated Yang-Baxter planar algebra and $P$ and $Q$ be its minimal idempotents orthogonal to $\begin{tikzpicture}[baseline = -.1cm, scale = .6]
    \filldraw[fill = black!40!] (-.5,-.5) to[out = 90, in = 180] (0,-.2) to[out = 0, in = 90] (.5,-.5);
    \filldraw[fill = black!40!] (-.5,.5) to[out = -90, in = 180] (0,.2)  to[out = 0, in = -90] (.5,.5);
\end{tikzpicture}.$ Then if the matrices $C_P$ and $C_Q$ give a non-symmetric spin model, they are the adjacency matrix of a regular tournament.
\end{proposition}

\begin{proof}
We first prove $C_P$ is the adjacency matrix of a directed graph. Most of the argument in the proof of Proposition \ref{prop:graph} was not dependent on the spin model being symmetric. Thus, we know that $C_P$ and $C_Q$ are matrices composed of 1s and 0s, they must have 0s on the diagonal, and that $C_P(a,b) = 1$ if and only if $C_Q(a,b) = 0$ and vice versa. Since the 2-click rotation of $P$ is $Q$, $c_P(a,b) = c_Q(b,a)$ for all $a$ and $b.$ 

Thus, by definition $C_P$ is the adjacency matrix of a directed graph where $a$ has an out-edge to $b$ if $C_P(a,b) = 1$ and $a$ has an in-edge to $b$ if $C_P(b,a) = 1.$ In addition, since
\begin{equation}\label{eq:qrelb}
 \begin{tikzpicture}[baseline = -.1cm]
    \draw[fill = black!40!, draw = none] (-.485,-.5) rectangle (.485,.5);
    \filldraw[fill = white] (.5,.5) to[out = -135, in = 90] (.2,0) to[out = -90, in = 135]  (.5,-.5);
    \filldraw[fill = white] (-.5,.5) to[out =-45, in = 90] (-.2,0) to[out = -90, in = 45] (-.5,-.5);
\end{tikzpicture} ~=~ \begin{tikzpicture}[baseline = -.1cm]
    \filldraw[fill = black!40!] (-.5,-.5) to[out = 90, in = 180] (0,-.2) to[out = 0, in = 90] (.5,-.5);
    \filldraw[fill = black!40!] (-.5,.5) to[out = -90, in = 180] (0,.2)  to[out = 0, in = -90] (.5,.5);
\end{tikzpicture} ~+~ \begin{tikzpicture}[baseline = -.1cm]
\draw[fill = black!40!, draw = none] (-.35,-.5) rectangle (.35,.5);
\draw (-.35,-.5) -- (-.35,.5);
\draw (.35,-.5) -- (.35,.5);
\draw[fill = white] (-.5,-.3) rectangle (.5,.3);
\node at (0,0) {\large{{$P$}}};
\end{tikzpicture} ~+~ \begin{tikzpicture}[baseline = -.1cm]
\draw[fill = black!40!, draw = none] (-.35,-.5) rectangle (.35,.5);
\draw (-.35,-.5) -- (-.35,.5);
\draw (.35,-.5) -- (.35,.5);
\draw[fill = white] (-.5,-.3) rectangle (.5,.3);
\node at (0,0) {\large{{$Q$}}};
\end{tikzpicture}  
\end{equation}
it must be the case that exactly one of $C_P(a,b)$ or $C_Q(a,b)$ is 1 for all distinct $a$ and $b$. Thus, for any distinct pair of vertices $a$ and $b$, $a$ has either an in-edge or out-edge to $b,$ which makes $C_P$ a tournament by definition.

By Relation 1b, we know that 
\begin{equation*}
\includegraphics[valign = c]{diagrams/pdf/YB1b2.pdf}  
\end{equation*}
which says that every vertex has $k$ out-edges. Capping $Q$ in a similar way and using (\ref{eq:qrelb}), we see that each vertex has $n-k-1$ in-edges. Thus, in total $\Gamma$ has $nk$ out-edges and $n(n-k-1)$ in-edges. Since every in-edge from $a$ to $b$ can be thought of as an out-edge from $b$ to $a,$ we know that $nk = n(n-k-1)$. Since $n > 0$, this implies that $k = n-k-1$ and that $\Gamma$ is a regular tournament, as desired.   
\end{proof}

Now that we are severely restricted by the type of directed graph that can give a non-symmetric spin model, we are ready to prove the following theorem, which is the main result of \cite{Jae95b}:

\begin{theorem}
Let $\Gamma$ be a directed graph. Then $\Gamma$ gives a non-symmetric spin model for a singly-generated Yang-Baxter planar algebra if and only if $\Gamma$ is the 3-cycle.
\end{theorem}

\begin{proof}
Let $\Gamma$ be a directed graph. Suppose $\Gamma$ gives a non-symmetric spin model. Then by Proposition \ref{prop:digraph}, we know that $\Gamma$ is a regular tournament. In order for $\Gamma$ to give a non-symmetric spin model, $k > 0$. Suppose that $k \ne 1$. Relation 3b tells us that
\begin{equation}\label{eq:nonsym3b}
\begin{split}
     \includegraphics[valign = c, scale = .95]{diagrams/pdf/shadedmiddlepatoms.pdf}  
   & =~ \includegraphics[valign = c, scale = .95]{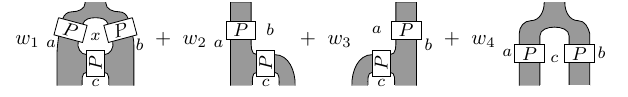} 
   \\
   & \includegraphics[valign = c, scale = .95]{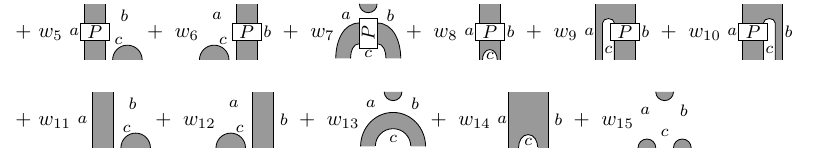} 
\end{split}
\end{equation}
for all vertices $a,$ $b,$ and $c.$ For any of the possible relationships among $a,$ $b,$ and $c$, we obtain a state sum for each diagram above. By Lemma \ref{lem:di} we know that $\Gamma$ contains subgraphs of the form
\begin{center}
    \includegraphics[valign = c]{diagrams/pdf/digraphtypes.pdf}
\end{center}
Thus, the state sums of the diagrams above give a system of 15 linear equations in terms of the $w_i.$ By inspection, though, this system is inconsistent. Thus, $\Gamma$ cannot give a spin model if $k \ne 1$. 

Now suppose that $k = 1$. Since $\Gamma$ is a regular tournament, $n = 2k + 1 = 3$. Since each vertex has one in-edge and one-out-edge, $\Gamma$ must be the 3-cycle. Assume that $\Gamma$ is the 3-cycle. Clearly, $\Gamma$ satisfies Relation 1a, 1b, and 2a. 

Note that 
\begin{equation*}
\mathlarger{\mathlarger{\sum}}_{x\in S} \begin{tikzpicture}[baseline = 0cm]
\draw[fill = black!40!, draw = none] (-.35,-.6) rectangle (1.85,.6);
\draw[fill = white] (.35,.3) to[out = 90, in = 180]  (.75,.55) 
							 to[out = 0, in = 90]    (1.15,.3)
					  		 to[out = -90, in = 90]  (1.15,-.3) 
					  		 to[out = -90, in = 0]   (.75, -.55)
					  		 to[out = 180, in = -90] (.35,-.3)
					  		 to[out = 90, in = -90]  (.35,.3);
\draw (-.35,-.6) -- (-.35,.6);
\draw[fill = white] (-.5,-.3) rectangle (.5,.3);
\draw (1.85,-.6) -- (1.85,.6);
\draw[fill = white] (1,-.3) rectangle (2,.3);
\node at (0,0) {\small{{$P$}}};
\node at (1.5,0) {\small{{$P$}}};
\node at (-.7,0) {\small{{$a$}}};
\node at (.7,0) {\small{{$x$}}};
\node at (2.2,0) {\small{{$b$}}};
\end{tikzpicture} ~=~  
\left\{
\begin{array}{cc}
    1 & a \ne b \\
    0 & a = b
\end{array}
\right.
\end{equation*}
Since 
\begin{equation*}
\begin{tikzpicture}[baseline = -.1cm]
    \draw[fill = black!40!, draw = none] (-.485,-.5) rectangle (.485,.5);
    \filldraw[fill = white] (.5,.5) to[out = -135, in = 90] (.2,0) to[out = -90, in = 135]  (.5,-.5);
    \filldraw[fill = white] (-.5,.5) to[out =-45, in = 90] (-.2,0) to[out = -90, in = 45] (-.5,-.5);
    \node at (-.5,0) {$a$};
    \node at (.5,0) {$b$};
\end{tikzpicture} ~-~ \begin{tikzpicture}[baseline = -.1cm]
    \filldraw[fill = black!40!] (-.5,-.5) to[out = 90, in = 180] (0,-.2) to[out = 0, in = 90] (.5,-.5);
    \filldraw[fill = black!40!] (-.5,.5) to[out = -90, in = 180] (0,.2)  to[out = 0, in = -90] (.5,.5);
    \node at (-.5,0) {$a$};
    \node at (.5,0) {$b$};
\end{tikzpicture} ~=~  
\left\{
\begin{array}{cc}
    1 & a \ne b \\
    0 & a = b
\end{array}
\right.
\end{equation*}
By Proposition \ref{prop:atoms}, we must have the relation
\begin{equation}\label{eq:3cyc}
    \begin{tikzpicture}[baseline = 0cm]
\draw[fill = black!40!, draw = none] (-.35,-.6) rectangle (1.85,.6);
\draw[fill = white] (.35,.3) to[out = 90, in = 180]  (.75,.55) 
							 to[out = 0, in = 90]    (1.15,.3)
					  		 to[out = -90, in = 90]  (1.15,-.3) 
					  		 to[out = -90, in = 0]   (.75, -.55)
					  		 to[out = 180, in = -90] (.35,-.3)
					  		 to[out = 90, in = -90]  (.35,.3);
\draw (-.35,-.6) -- (-.35,.6);
\draw[fill = white] (-.5,-.3) rectangle (.5,.3);
\draw (1.85,-.6) -- (1.85,.6);
\draw[fill = white] (1,-.3) rectangle (2,.3);
\node at (0,0) {\small{{$P$}}};
\node at (1.5,0) {\small{{$P$}}};
\end{tikzpicture} ~=~ \begin{tikzpicture}[baseline = -.1cm]
    \draw[fill = black!40!, draw = none] (-.485,-.5) rectangle (.485,.5);
    \filldraw[fill = white] (.5,.5) to[out = -135, in = 90] (.2,0) to[out = -90, in = 135]  (.5,-.5);
    \filldraw[fill = white] (-.5,.5) to[out =-45, in = 90] (-.2,0) to[out = -90, in = 45] (-.5,-.5);
\end{tikzpicture} ~-~ \begin{tikzpicture}[baseline = -.1cm]
    \filldraw[fill = black!40!] (-.5,-.5) to[out = 90, in = 180] (0,-.2) to[out = 0, in = 90] (.5,-.5);
    \filldraw[fill = black!40!] (-.5,.5) to[out = -90, in = 180] (0,.2)  to[out = 0, in = -90] (.5,.5);
\end{tikzpicture}
\end{equation}
Hence, Relation 2b is satisfied.

Since the 3-cycle has only one in-edge and out-edge for each vertex 
\begin{equation*}
     \includegraphics[valign = c]{diagrams/pdf/shadedmiddlepatoms.pdf}  ~=~
     \includegraphics[valign = c]{diagrams/pdf/unshadedmiddlepatoms.pdf}~=~ 0
\end{equation*}
for all $a,$ $b$, and $c$. Thus, by Proposition \ref{prop:atoms}
\begin{equation*}
     \includegraphics[valign = c, scale = .5]{diagrams/pdf/shadedmiddlep.pdf}  ~=~
     \includegraphics[valign = c, scale = .5]{diagrams/pdf/unshadedmiddlep.pdf}~=~ 0
\end{equation*}
and so Relations 3a and 3b are satisfied. Thus, $\Gamma$ gives a spin model for $\mathcal{V}$ if and only if it is the 3-cycle, as desired.
\end{proof}

Similar to the symmetric case, using the classification of singly-generated Yang-Baxter planar algebras in \cite{Liu15}, we can make the following corollary:

\begin{corollary}
Let $\Gamma$ be the 3-cycle. Then $\Gamma$ gives a spin model for the Bisch-Jones planar algebra with $\dim V_3 = 9.$
\end{corollary}

\begin{proof}
Let $\mathcal{V}$ be the planar algebra with spin model given by $\Gamma$. By inspection, $\dim V_3 = 9.$ By the classification of singly-generated Yang-Baxter planar algebras and Proposition \ref{inj}, there is only one such planar algebra, which is a Bisch-Jones planar algebra. Specifically, it is the $\mathbb{Z}/3$-group planar algebra found in \cite{BJ00}.
\end{proof}

\label{class}

\bibliographystyle{alpha}
\bibliography{text/DissBib}

\end{document}